\documentclass[11pt,letterpaper]{amsart}
\usepackage[margin=0.75in]{geometry}
\usepackage{amsmath,amsfonts,amssymb} 
\usepackage{algpseudocode}
\usepackage{algorithm}
\usepackage{seqsplit}
\usepackage{subcaption}
\usepackage{lipsum}
\usepackage{numprint}
\usepackage{bigintcalc}
\usepackage{mathtools}
\usepackage{mathrsfs}
\usepackage{graphicx}
\usepackage{xcolor}
\usepackage{titling}
\newcommand{\headers}[2]{} 
\usepackage{amsthm} 
\newtheorem{theorem}{Theorem}[section]
\newtheorem{corollary}[theorem]{Corollary}
\newtheorem{definition}[theorem]{Definition}
\newtheorem{conjecture}[theorem]{Conjecture}
\newtheorem{lemma}[theorem]{Lemma}

\newcommand{\newsiamremark}[2]{} 
\newcommand{\eps}{\varepsilon}
\definecolor{myblue}{rgb}{0,0,1}
\definecolor{mypurple}{rgb}{0.5,0,0.5}
\usepackage{hyperref}
\hypersetup{
    colorlinks=true,
    linkcolor=myblue,
    citecolor=mypurple,
    urlcolor=cyan
}

\newcommand{\ignore}[1]{}

\algnewcommand{\LineComment}[1]{\State \(\triangleright\) #1}

\def\R{\mathbb{R}}
\def\Q{\mathbb{Q}}

\def\Z{\mathbb{Z}}

\def\S[#1]{S^{#1}}

\def\ra{\rightarrow}

\def\sgn{\mathrm{sgn \, }}

\def\R{\mathbb{R}}
\def\Q{\mathbb{Q}}

\def\eps{\varepsilon}
\def\Z{\mathbb{Z}}

\begin{document}

  \begin{center}
    \vspace*{-1.0em}
    {\Large \bfseries Variants on the $abc$-Conjecture using Alternative Quality Metrics$^{1}$}\\
    \vspace{0.8em}
    
    {\large Akilan Sankaran$^{2}$}\\
    \vspace{0.3em}
    {\small Harvard University, Cambridge, MA 02138}\\
    \vspace{0.6em}
    {\small \today}\\
    \vspace{1.2em} 

    {\normalsize\bfseries\scshape {ABSTRACT}}
    \vspace{0.4em}
    
    \begin{minipage}{0.92\textwidth}
      \small
      The $abc$-conjecture (Masser and Oesterl\'e) has remained open for decades. By measuring $abc$-triples using a particular \textit{quality metric}, the conjecture seeks the asymptotic distribution of triples of sufficient quality. We create new classes of quality metrics to develop variants on the $abc$-conjecture, with each metric based upon the doubly geometric mean of the prime factors of triples. We investigate the behavior of the resulting class of quality metrics; by determining families of triples that yield high quality, we establish several asymptotic results that are analogous to the $abc$-conjecture for our metrics. We also establish sharp phase transitions for the behavior of families of such quality metrics within specified parametrizations for smoothness of primes in $abc$-triples, using heuristics from the Szpiro ratio for associated Frey curves. Finally, we implement algorithms to determine triples with high qualities with sub-linear runtime, an asymptotic speedup over naïve approaches. Our analysis offers robust variations of, and connections to, the $abc$-conjecture that offer independent questions of analytical interest.
    \end{minipage}
    \vspace{2em} 
  \end{center}

%footnotes
\footnotetext[1]{May 2026 update to mss.\ includes Theorems \ref{T_Chen}–\ref{T_abc_eta}. Results in Sec. \ref{S_DGM_Quality} were formulated in 2023.}
\footnotetext[2]{Correspondence: \href{mailto:akilansankaran@college.harvard.edu}{akilansankaran@college.harvard.edu}. I am grateful for the guidance of Dr.\ David Metzler at Albuquerque Academy throughout this effort.}
\setcounter{footnote}{2}

\section{Introduction}

The $abc$-conjecture, introduced by Joseph Oesterl\'e \cite{oesterle1988} and David Masser \cite{masser1985} in the 1980s, has remained an active conjecture in number theory. The conjecture asserts an asymptotic result on what we denote the \textit{quality} $q(\cdot)$ of triples $(a, b, c)$ with $a + b = c$, for coprime $a, b, c$; such triples are denoted $abc$-triples. The significance of the $abc$-conjecture extends far beyond an isolated statement; a proof would immediately imply Fermat's Last Theorem for $n\ge 6$ \cite{wiles1995, waldschmidt2015}, the Fermat--Catalan and Beal conjectures, and Szpiro's conjecture on elliptic curves \cite{waldschmidt2015}, while also shedding light on the distribution of prime gaps. Outside pure mathematics, deeper understanding of prime structure underlies cryptographic primitives (e.g.\ RSA), efficient integer factorization (the General Number Field Sieve depends on the abundance of smooth numbers \cite{gnfs}), and even biological phenomena – such as the prime-numbered life cycles of cicadas. 

Several recent works have established theoretical results that provide insight towards the conjecture itself \cite{granville2002} and developed computer algorithms to find high-quality triples quickly \cite{desmit2022, nitaj1993-algorithms, nitaj1993-an}. Nitaj characterized several of the relevant methodologies for triple searches \cite{nitaj2022}; van der Horst introduced an algorithm using elliptic curves \cite{horst2010}; and Elkies illustrated connections to other unsolved problems in number theory \cite{elkies2007}. Given the intractable nature of determining high-quality triples under the standard \textit{quality metric} to evaluate the suitability of $abc$-triples as measured by the conjecture, a natural strategy is to pursue \emph{alternative} quality metrics that preserve the conjecture's qualitative content while being more tractable from analytic standpoints. This approach is the focus of the ensuing manuscript; we note that previous variants of the quality used in the $abc$-conjecture can be found in ratios involving conductors and discriminants of associated elliptic curves, as well as minor algebraic variations \cite{waldschmidt2015, horst2010, martin2016}. Our metrics, which are based on the Doubly Geometric Mean of the primes dividing $abc$ for each such $abc$-triple ($a, b, c$), assume a distinct approach by directly penalizing the number $\omega$ of distinct prime divisors of $abc$.

The main contributions of this work are as follows:
\begin{enumerate}
    \item We define several new quality metrics (Definitions \ref{D_DGM_Quality}, \ref{D_DGM_class}) in the same spirit as the standard quality for the $abc$-conjecture, which (i) penalize the number of distinct prime divisors of $abc$, and (ii) incorporate roundness of an $abc$-triple.
    \item We analyze all variant quality metrics by constructing families of high-quality triples and investigating asymptotic growth rates (Theorems \ref{T_2_3}, \ref{T_power_of_p_q}, \ref{T_3_4}.) We develop an analogue of the $abc$-conjecture (Theorem \ref{C_abc_DGM}) and prove it conditionally on the infinitude of twin primes or Mersenne primes (Theorem \ref{T_twin_DGM}), and \emph{unconditionally} through Chen's theorem (Theorem \ref{T_Chen_DGM}). 
    \item We also prove results on analogues of transition phases with respect to asymptotic behavior for parametrized families of quality metrics in Theorems \ref{T_phase_transition}, \ref{T_uncond_LB}, and \ref{T_cond_UB} – and develop connections to elliptic curves and the $abc$-conjecture itself in Theorems \ref{T-szpiro}, \ref{t-big}, \ref{T_abc_eta}, and \ref{T_qt_UB}.
    \item We show an asymptotic improvement on the brute force algorithm (Algorithm \ref{A_Naive}) via constructive enumerations of high-quality triples (Algorithms \ref{A_2_3}, \ref{A_cyclo}), enabling scalable, sublinear-time computation of high-quality triples under our alternative quality metrics.
\end{enumerate}

\section{Background}
\label{S_Background}

We first offer a brief description of the frame of reference from which we approach the $abc$-conjecture, as a manner of motivating our discussions in Section \ref{S_DGM_Quality}.

\iffalse
% [CUT: original §2.1 motivation chain: nonlinear Diophantine equation,
%  Fermat's Last Theorem statement, Fermat's Last Theorem for polynomials,
%  Mason-Stothers theorem, integer analogue, test quality definition,
%  Theorem \ref{T_test_quality} that test quality is unbounded.
%  This entire motivation chain is condensed to one paragraph below.]

\subsection{Motivation for Quality Metric}
We can motivate the idea of a quality metric of an $abc$-triple through Fermat's Last Theorem. Consider the nonlinear Diophantine equation
\begin{align}\label{E_Fermat_Equation} x^n + y^n = z^n \end{align}
where $x, y, z, n \in \mathbb{Z}$. [... full text retained in original; see \cite{wiles1995, granville2002, mason1984, stothers1981} for the chain Fermat $\to$ Mason-Stothers $\to$ integer analogue $\to$ quality ...]
\begin{theorem}[Fermat's Last Theorem]\label{T_FLT}\end{theorem}
\begin{theorem}[Mason-Stothers]\label{T_Mason}\end{theorem}
\begin{definition}[Test Quality]\label{D_Test_Quality}\end{definition}
\begin{theorem}\label{T_test_quality}\end{theorem}
\fi

\subsection{From Mason-Stothers to the Standard Quality}

The $abc$-conjecture arises as the integer analogue of the Mason--Stothers Theorem \cite{mason1984,stothers1981}. The latter result asserts that if there exist nonconstant polynomials $x(t),y(t),z(t)\in\mathbb{C}[t]$ sharing no common roots, such that $x(t)+y(t)=z(t)$, then $\max\{\deg x,\deg y,\deg z\}<|\{\mu\in\mathbb{C}:(xyz)(\mu)=0\}|$. The integer analogue of this conjecture thereby replaces the degree $\deg$ by the absolute value, and the count of distinct roots $\mu$ of $(xyz)(t)$ by the \textit{radical} of a given triple. The following definitions arise.
\begin{definition}[$abc$-triple]\label{D_abc_triple}
A triple $(a,b,c) \in (\Z^{+})^3$ is an \emph{$abc$-triple} if, and only if, $a+b=c$ and $\gcd(a,b)=\gcd(b,c)=\gcd(a,c)=1$.
\end{definition}
\begin{definition}[Radical]\label{D_radical}
The \emph{radical} of an integer $n\in\Z^{+}$ is given by $\mathrm{rad}(n) : =\prod_{p\mid n,\,p\in\mathbb{P}} \, p$.
\end{definition}
A first attempt at a quality metric is therefore of the form $q_{\text{test}}(a,b,c)=c/\mathrm{rad}(abc)$, which measures the \textit{roundness} of the triple $(a, b, c)$ as per the size of the largest entry, $c$. However, since the family $(1,2^{p(p-1)n}-1,2^{p(p-1)n})$ for $p\in\mathbb{P}$, $n\in\Z^{+}$ has $\mathrm{rad}(abc)<2c/p$ (by Fermat's Little Theorem), it follows that $q_{\text{test}}$ may become arbitrarily large. To produce a more analytically interesting metric, one might reduce the growth rate of the quality by introducing a logarithmic term in the numerator, developing the following.\footnote{We designate this quality as the \textit{standard quality} since it has been analyzed extensively and is the focal point of the conjecture itself \cite{waldschmidt2015, desmit2022, horst2010, martin2016}. We restrict to $abc$-triples (with $a,b,c$ coprime); otherwise triples of the form $(2^h,2^h,2^{h+1})$ have unbounded quality.}
\begin{definition}[\textbf{Standard Quality}]\label{D_standard_quality}
For an $abc$-triple $(a, b, c)$, the standard quality is
\[
    q_s(a, b, c) = \frac{\ln(c)}{\ln(\mathrm{rad}(abc))}.
\]
\end{definition}

To our knowledge, the tightest uniform bound on the asymptotics of the quality $q_s(\cdot)$ for large triples is the following.
\begin{theorem}[Stewart and Yu, 2001 \cite{stewart2001}]\label{T_Stewart}
There exists some fixed $C_0$ such that
$
c < \exp\bigl(C_0\sqrt[3]{\mathrm{rad}(abc)}\,(\ln \mathrm{rad}(abc))^{3}\bigr)
$
for all $abc$-triples.
\end{theorem}

Taking exponentials reveals the bound $q_s(a, b, c) \leq C_0 (\mathrm{rad}(abc))^{1/3} (\ln(\text{rad}(abc))^{1/3}$, which implies that the quality is bounded above by a sublinear function in the radical $\text{rad}(abc)$.

Only 241 triples with $q_s>1.4$ are known \cite{desmit2022}, with the highest quality obtained being $q_s\approx 1.6299$ (due to Ressyat, per continued fractions \cite{waldschmidt2015,horst2010}). See Table \ref{tab1} for a few such examples. We denote triples with quality $q_s>1$ to be \emph{$abc$-hits}; those with $q_s>1.4$ are \emph{high-quality}, per the classification given by the standard quality. We note, additionally, that the triple with fifth-highest known quality, $q_s$, is of the form $(7^3, 3^{10}, 2^{11} \cdot 29)$; this triple utilizes a remarkably smaller number of primes than the other highest-ranking triples, under the standard quality \cite{deweger1989}. This principle of prioritizing few primes in the coupled factorization of $abc$ motivates our alternative metrics in Section \ref{S_DGM_Quality}.

\iffalse
% [CUT: long-form abc-conjecture exposition, including
%  Theorem \ref{T_standard_infinite_values_bigger_than_1} (q_s>1 for infinitely
%  many triples), Corollary \ref{C_limsup_quality}, two alternate forms,
%  full Fermat's Last Theorem-from-abc proof (Conjecture \ref{C_strong_abc} and
%  consequent FLT theorem). Conjecture \ref{C_abc_conjecture} retained inline below.]
\fi

\subsection{The $abc$-Conjecture}
Having motivated the standard quality metric $q_s(a, b, c)$ itself, we now present the statement of the $abc$-conjecture.

\begin{conjecture}[$abc$-conjecture \cite{oesterle1988, masser1985}]\label{C_abc_conjecture}
For any real $\varepsilon > 0$, there exist only finitely many $abc$-triples $(a,b,c)$ with quality $q_s(a, b, c) > 1 + \varepsilon$. Equivalently, $\lim\sup_{c \rightarrow \infty} q_s(a, b, c) = 1$.
\end{conjecture}

Via the previously described family of $abc$-triples \newline $(1, 2^{p( p - 1)n} - 1, 2^{p(p - 1)n})$, for $p$ prime, it is known that the quality $q_s$ satisfies $q_s(a,b,c)>1$ for infinitely many such $abc$-triples (\cite{waldschmidt2015}). However, the $\limsup$ criterion remains open; additionally, it remains to be established that \newline $q_s(a, b, c) \leq 2$ over all such triples.  This bound – which corresponds to the $\eps = 1$ case in Conjecture \ref{C_abc_conjecture} above – is known to imply Fermat's Last Theorem for $n \ge 6$, for a primitive solution to the Diophantine equation $x_0^n+y_0^n=z_0^n$ would yield $q_s\ge n/3$, contradicting $q_s \leq 2$ \cite{wiles1995, waldschmidt2015}.

\begin{table*}[t]
    \centering
    \begin{tabular}{|c|c|c|c|c|}
    \hline
        $a$ & $b$ & $c$ & \textbf{Quality $q_s$} & \textbf{Creator} \\
    \hline
        $2$ & $3^{10} \cdot 109$ & $23^5$ & $1.6299$ & Ressyat \cite{waldschmidt2015}\\
        $11^2$ & $3^2 \cdot 5^6 \cdot 7^3$ & $2^{21} \cdot 23$ & $1.6260$ & de Weger \cite{deweger1989}\\
        $19 \cdot 1307$ & $7 \cdot 29^2 \cdot 31^8$ & $2^8 \cdot 3^{22} \cdot 5^4$ & 1.6235 & Browkin, Brzezinski \cite{browkin1994} \\
        $283$ & $5^{11} \cdot 13^2$ & $2^8 \cdot 3^8 \cdot 17^3$ & 1.5808 & Browkin, Brzezinski, Nitaj \cite{nitaj1993-algorithms, browkin1994} \\
        $1$ & $2 \cdot 3^7$ & $5^4 \cdot 7$ & 1.5679 & de Weger \cite{deweger1989} \\
        $7^3$ & $3^{10}$ & $2^{11} \cdot 29$ & 1.5471 & de Weger \cite{deweger1989}\\
        $7^2 \cdot 41^2 \cdot 311^3$ & $11^{16}\cdot 13^2 \cdot 79$ & $2\cdot 3^3\cdot 5^{23}\cdot 953$ & 1.5444 & Nitaj \cite{nitaj1993-algorithms} \\ \hline
    \end{tabular}
    \caption{Several of the known high-quality triples under the standard quality $q_s$ (see Definition \ref{D_standard_quality}).}
    \label{tab1}
\end{table*}

\iffalse
% [CUT: Original Section 2.4 (Properties of Standard Quality) with full
%  prime-factorization expansion of q_s leading to Equation \ref{eq1}
%  and \ref{GMsq}. Condensed below to retain only the final GM form,
%  which is what motivates the DGM quality.]
\fi
We may also view the standard quality metric as an effective geometric mean of the primes dividing $abc$, when each prime is provided with the corresponding weight.

For instance, we may write the prime factorization of $abc$ with distinct primes $p_1,\dots,p_\omega$, such that $\omega$ is the number of primes dividing $abc$. Manipulation yields
\begin{align}
\label{eq1}
    q_s(a, b, c) = \frac{\sum_{p_i  \mid c} e_i \ln(p_i)}{\sum_{i = 1}^{\omega} \ln(p_i)},
\end{align}
where we write $c = \prod p_i^{e_i}$ for all such primes $p_i$ dividing $c$. Recognizing the resulting $\omega$-th root as a geometric mean, we observe
\begin{align}
\label{GMsq}
    q_s(a, b, c) = \frac{\ln(c)}{\ln\!\bigl(\,\sqrt[\omega]{\prod_{i = 1}^{\omega} p_i^{\omega}}\,\bigr)} = \frac{\ln(c)}{\ln\!\bigl(\mathrm{GM}\{p_i^\omega\}\bigr)},
\end{align}
where $\mathrm{GM}(\cdot)$ denotes the geometric mean of the given entities. Thus, it follows that the standard quality is a comparison of the benchmark size measure $\ln(c)$ to the size measure provided by the geometric mean of the primes dividing $abc$, each weighted by the corresponding $\omega$. Thus, one might seek to define a quality metric that seeks to penalize the \textit{number} of primes much more harshly than their size, effectively selecting \textit{against} smoothness of $abc$-triples. This effort is the task of the following section.

\section{DGM Quality Metric}
\label{S_DGM_Quality}

\subsection{Motivation and Definition}

Observe, in Table \ref{tab1}, that many high-quality triples (for instance, the third highest-quality triple of $(19 \cdot 1307, 7 \cdot 29^2 \cdot 31^8, 2^8 \cdot 3^{22} \cdot 5^4)$) use significantly more primes than those of the first, fifth, and sixth highest quality triples. Notably, the characterization of the standard quality in Definition \ref{D_standard_quality} guarantees that quality metrics that may use a large number of primes, yet ensure that the radical $\text{rad}(abc)$ remains small through comparatively large exponents on primes of small magnitude, might be prioritized over triples that utilize a small number of (possibly large) primes. Thus, we seek a metric that aggressively privileges a small \textit{overall} number of primes within the factorization of $abc$.\footnote{For instance, the cardinality of the solution set to some Diophantine equations may be bounded specifically with respect to a function of the number of prime factors of a given integer in consideration. Explicitly, the Thue equation $F(x, y) = m$ (for a homogeneous form $F$, irreducible over $\Q$ with degree $d$ greater than $3$) has solutions bounded by $O(d^{\omega(m) + 1})$, where $\omega(m)$ is the number of primes in the factorization of $m$. If $m$ represents $abc$ in our case, an analogy may be drawn.}

Our solution invokes the doubly geometric mean of such primes dividing $abc$ in each $abc$-triple. As motivation, consider the sequence of means with progressively stronger \textit{taming} of large values; that is, if we write the sequence of means as
\begin{align*}
\mathrm{AM}(x_i) &:= \tfrac{1}{k}\sum_{i = 1}^{k} x_i,\\
\mathrm{GM}(x_i) &:= \exp\bigl(\mathrm{AM}\{\ln x_i\}\bigr),\\
\mathrm{DGM}(x_i) &:= \exp\bigl(\exp(\mathrm{AM}\{\ln\ln x_i\})\bigr),
\end{align*} 
then each step progressively adds an inner logarithmic factor and an outer exponential, which rescales such that all three means agree when the $x_i$ are all identical. For instance, on the set $\{2,3,5,7,1009\}$, the arithmetic mean $205.2$ is dominated by the single large prime $1009$ (for removing it drops the AM to $4.25$), while the GM is $24$ and the DGM is $5.04 \approx 5$. Thus, the DGM rewards triples that avoid many primes in their factorizations with lower prioritization of size, which isolates $\omega$ (the number of primes dividing $abc$) as the dominant parameter, to a greater extent than the partial role of $\omega$ in the standard quality metric \cite{deweger1989}. Explicitly, we have the following.
\begin{definition}[Doubly Geometric Mean]\label{D_DGM_Mean}
For a set $S=\{x_1,\dots,x_k\}\subset\mathbb{R}_{>0}$, we evaluate
\[
    \mathrm{DGM}(S) := \exp\Big(\exp\Big(\tfrac{1}{k} \sum_{i = 1}^{k} \ln\ln(x_i)\Big)\Big).
\]
\end{definition}
Using the DGM metric, the resulting \textit{DGM Quality} may be defined in the natural extension of equation (\ref{GMsq}).
\begin{definition}[DGM Quality]\label{D_DGM_Quality}
Over all $(a,b,c) \in (\Z^{+})^3$ an $abc$-triple, and $\omega$ being the number of primes dividing $abc$, with $P_\omega=\{p^\omega:p\in \mathbb{P}, p \mid abc \}$, we define the quality $q_{DGM}: (\Z^{+})^3 \mapsto \R^{+}$ by
\[
    q_{\mathrm{DGM}}(a, b, c) := \frac{\ln(c)}{\ln(\mathrm{DGM}(P_{\omega}))}.
\]
\end{definition}
Using Definition \ref{D_DGM_Mean}, the DGM quality may be rewritten in a suggestive manner. For a given $abc$-triple $(a, b, c)$ with corresponding set $P_{\omega}$ as given in Definition \ref{D_DGM_Quality} above, we have
\begin{align*}
q_{\mathrm{DGM}}(a, b, c) &= \frac{\ln(c)}{\exp\!\bigl(\tfrac{1}{\omega}\sum_i \ln\ln(p_i^\omega)\bigr)} \\
&= \frac{\ln(c)}{\exp\!\bigl(\ln\omega + \tfrac{1}{\omega}\sum_i \ln\ln p_i\bigr)},
\end{align*}
which may be rewritten as
\begin{align}\label{eq3}
    q_{\mathrm{DGM}}(a,b,c) = \frac{\ln(c)}{\omega \cdot \sqrt[\omega]{\prod_{i = 1}^{\omega} \ln(p_i)}}.
\end{align}

We remark, upon comparing equation (\ref{eq3}) with that of (\ref{eq1}), that while both effectively involve a geometric mean of the prime divisors of $abc$ as a \textit{penalty} for each triple, the quality $q_{\mathrm{DGM}}$ includes an explicit multiplicative contribution of $1/\omega$. Thus, the DGM quality penalizes triples with many prime divisors overall with a linearly dependent penalty on $\omega$. Additionally, the contribution of each prime towards the geometric mean is \textit{logarithmic}, rather than a contribution of $p_i^\omega$ within the standard quality. Thus, the DGM quality directly prioritizes triples with a smaller overall number of primes in their factorizations.

\subsection{Elementary Bounds and Families of High-Quality Triples}
\label{S_families}

We begin by establishing a series of straightforward bounds on the DGM Quality, to contextualize later results on asymptotics. We begin by establishing the following lemma.

\begin{figure}[t]
\centering
\includegraphics[width=0.7\columnwidth]{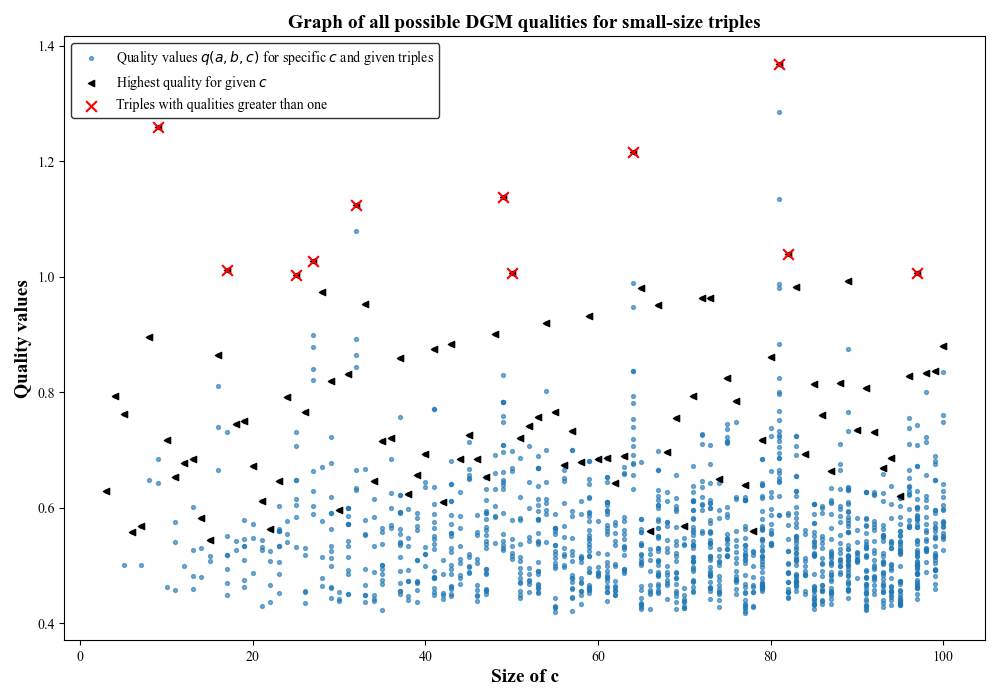}
\caption{$q_{\mathrm{DGM}}$ evaluated on all $abc$-triples $(a, b, c)$ with $c \le 100$. Red “x"s are DGM-hits, and black “x"s give the highest-quality triples for a valid $abc$-triples upon a given value of $c$.\label{fig1}}
\end{figure}

\begin{lemma}\label{L_ln_primes_DGM}
For any fixed integer $s\in\Z^{+}$ with distinct prime divisors of the form $p_1,\dots,p_{\omega_s}$, we must have that
$\sum_{i=1}^{\omega_s}\ln(p_i)\le\ln(s)$, with equality if, and only if, $s$ is a squarefree integer.
\end{lemma}
\begin{proof}
Write the prime factorization of $s$ as $\prod_{i = 1}^{\omega_s} p_i^{e_i}$, with all exponents $e_i$ satisfying $e_i \ge 1$. Then, it follows that $\ln(s) = \sum_{i = 1}^{\omega_s} e_i \ln(p_i)$. Hence, we must have that $\ln s=\sum e_i\ln p_i\ge\sum\ln p_i$ since all the $e_i \ge 1$.
\end{proof}
Therefore, if we are able to convert the expression in equation (\ref{eq3}) to one involving a \textit{sum} of the logarithm of the primes dividing $abc$, the above result guarantees an upper bound on the denominator. The AM-GM inequality provides precisely this tool.
\begin{theorem}\label{T_13}
For every $abc$-triple, $q_{\mathrm{DGM}}(a, b, c) > \tfrac{1}{3}$.
\end{theorem}
\begin{proof}
Let $(a, b, c)$ an $abc$-triple, with $\omega$ defined as before. We apply the AM-GM inequality on equation (\ref{eq3}):
\begin{align*}
q_{\mathrm{DGM}} (a, b, c) &\ge \frac{\ln(c)}{\omega \cdot \frac{1}{\omega} \sum_{i = 1}^{\omega} \ln(p_i)} \ge\frac{\ln(c)}{\ln(abc)},
\end{align*}
where the second inequality invokes Lemma \ref{L_ln_primes_DGM}. Since $a, b < c$, we have 
\begin{align*}
    q_{\mathrm{DGM}} (a, b, c) > \frac{\ln(c)}{\ln(c^3)} = \frac{1}{3}.
\end{align*}
\end{proof}

As in the case for the standard quality, we state that an $abc$-triple is a \textit{DGM hit} if we have $q_{\text{DGM}}(a, b, c) > 1$; per Definition \ref{D_DGM_Quality}, this is equivalent to the condition $\mathrm{DGM}(P_\omega)<c$. Figure \ref{fig1} shows $q_{\mathrm{DGM}}$ evaluated on all $abc$-triples with $c\le 100$, with DGM hits isolated. 

For any positive integer $c$, we define the sequence $s_c := \max\{q_{\mathrm{DGM}}(a,b,c):(a,b,c)\text{ is an }abc\text{-triple}\}$ at fixed $c$. We now seek to investigate $\limsup s_c$ via various families of triples; this is suggested by certain subsets of triples in Figure \ref{fig1} that appear to ascend roughly logarithmically in the size of $c$.

Upon inspection many of the high-quality triples $(a, b, c)$ in Figure \ref{fig1} utilize components in $a, b$ with maximally simple factorization – for instance, powers of $2$ or $3$. For instance, we may set $a=2^i$ and $b=3^j$ for indices $i,j$; if we have that $c=a+b$ is prime, then $\omega=3$. Consequently, equation (\ref{eq3}) immediately gives
\begin{align*}
q_{\mathrm{DGM}}(2^i,3^j,c) &= \frac{\ln c}{3({\ln 2\cdot\ln 3\cdot\ln c})^{1/3}}\\
&= (\ln c)^{2/3}\cdot\frac{1}{3\sqrt[3]{\ln 2\cdot\ln 3}},
\end{align*}
which diverges asymptotically as a function of $c$. We designate these triples to constitute the \textit{power of 2, 3} family; the asymptotic distribution of triples of this form influences the  behavior of the DGM quality, as follows.
\begin{theorem}\label{T_2_3}
If there exist infinitely many pairs $(i,j) \subset (\Z^{+})^2$ such that $c=2^i+3^j$ is prime, then $\limsup_{c\to\infty}s_c=\infty$.
\end{theorem}
\begin{proof}
Let $N\in\Z^{+}$, and pick $c_0>\exp(3N\sqrt{3N})$. Then, by the above bound, $s_c \ge q_{\mathrm{DGM}}(a, b, c)>\frac{1}{3}\cdot (\ln(c_0))^{2/3}>N$ for all \textit{power of 2, 3} triples $(a, b, c)$ with $c>c_0$.
\end{proof}
This method may be extended to establish the following theorem in greater generality, with a slight modification of the relevant bound.
\begin{theorem}\label{T_power_of_p_q}
For distinct primes $p,q$, if there are infinitely many pairs $(i,j)$ with $c=p^i+q^j$ prime, then
$q_{\mathrm{DGM}}(p^i,q^j,c)=(\ln c)^{2/3}/\sqrt[3]{3\ln p\ln q}\to\infty$, such that $\limsup s_c=\infty$.
\end{theorem}

\iffalse
% [CUT: §3.3.3 Fixed Sequence of Primes Method (Theorem \ref{T_main}). The
%  result generalizes \ref{T_2_3}/\ref{T_power_of_p_q} to any finite
%  fixed prime sets; conclusion q_DGM ~ C (ln c)^{1 - 1/(m+n+1)} -> infty.
%  Replaced by a remark.]
\fi

The previous two constructions are special cases of the following general theorem, which allows the integers $a,b$ in an $abc$-triple to be more general products of powers over \textit{fixed} sets of primes.

\begin{theorem}[Fixed sequence of primes]\label{T_main}
For fixed primes $a_1,\dots,a_n$ and $b_1,\dots,b_m$ and exponents $i_1,\dots,i_n,j_1,\dots,j_m$ in $\Z^{+}$, let $a=\prod_{k=1}^n a_k^{i_k}$, $b=\prod_{k=1}^m b_k^{j_k}$, $c=a+b$. If there are infinitely many tuples $(i_1,\dots,i_n,j_1,\dots,j_m)$ such that $c$ is prime, then $\limsup_{c\to\infty}s_c=\infty$.
\end{theorem}
\begin{proof}
Let $\left( \prod_{k = 1}^n a_k^{i_k}, \prod_{k = 1}^{m} b_k^{j_k}, c\right)$ be such a valid $abc$-triple. For such triples, we have $\omega=m+n+1$, and equation (\ref{eq3}) gives
\begin{align*}
q_{\mathrm{DGM}}(a,b,c) &= \frac{\ln c}{(m\!+\!n\!+\!1)\left( \ln(c) \cdot {\prod_{i=1}^{m+n} (\ln (p_i))}\right)^{1/{m + n + 1}}}\\
 &= (\ln c)^{1-\frac{1}{m+n+1}}\cdot K,
\end{align*}
where $K=\frac{1}{m + n + 1}\bigl(\prod_{i=1}^{m+n}\ln p_i\bigr)^{-1/(m+n+1)}$ is independent of the choice of exponents $i_1, \dotsc, i_n, j_1, \dotsc, j_m$. Since $1-1/(m+n+1)>0$, the logarithmic term, and hence $q_{\mathrm{DGM}}\to\infty$, diverge along the infinite family of $abc$-triples as $c \ra \infty$.
\end{proof}
Theorems \ref{T_2_3} and \ref{T_power_of_p_q} thereby give the cases $m = n = 1$ for the above result, with the rate of growth decreasing as more primes are admitted (and $m,n$ grow).

Such methods of searching for triples with concise factorizations also yield other families of the form $(1, c - 1, c)$ for various special choices of $c$. For instance, for a given Fermat prime $p = 2^{n} + 1$, a \textit{Fermat triple} of the form $(1, 2^n,2 ^{n + 1})$ may be formed, such that $\omega = 2$. An analogous derivation to Thm. \ref{T_2_3} for this value of $\omega$ gives $q_{DGM}(1, 2^n, 2^{n + 1}) \ge \sqrt{n}/2$ for all such triples. Hence, divergent asymptotic behavior occurs, conditionally on the infinitude of Fermat primes, but only five such primes are known \cite{tsang2010}).
\subsubsection{Mersenne Triples}
A natural extension of this approach, however, is to invoke Mersenne primes (of the form $2^n - 1$) rather than Fermat primes. For all primes of the form $2^n - 1$, we must have that $n$ itself is prime, creating \textit{Mersenne triples }of the form $(1, 2^{n} - 1, 2^{n})$ for suitable prime $n, 2^n - 1$ \cite{mersenne}. There are 51 known Mersenne primes; the largest known prime overall is Mersenne, with $n=82{,}589{,}933$ \cite{mersennelist}. In an analogous manner to the Fermat triple case, a bound on $q_{\text{DGM}}$ may be established for Mersenne triples.

\begin{theorem}\label{T_3_4}
For all Mersenne triples $(1,2^n-1,2^n)$, $q_{\mathrm{DGM}}(a,b,c)=\sqrt{n}/2$.
\end{theorem}
\begin{proof}
Let $(1, 2^{n} - 1, 2^n)$ be a Mersenne triple, such that $n, 2^{n} - 1$ are both prime. Then, we must have that $\omega = 2$, and hence, applying (\ref{eq3}),
\begin{align*}
q_{\mathrm{DGM}} = \frac{\ln c}{2\sqrt{\ln 2\cdot\ln c}}
&=\sqrt{n\ln 2}\cdot\frac{1}{2\sqrt{\ln 2}}=\frac{\sqrt{n}}{2}. \qedhere % added extra
\end{align*}
\end{proof}
Thus, for the largest known Mersenne prime, which corresponds to $n_m=82{,}589{,}933$, we have
\[
    q_{\mathrm{DGM}}(1,\,2^{n_m}-1,\,2^{n_m}) = \frac{\sqrt{82{,}589{,}933}}{2} \approx 4543.95,
\]
which is the triple corresponding to the highest DGM quality that we found using the Mersenne method.

Therefore, upon the Lenstra-Pomerance-Wagstaff conjecture that there exist an infinitude of Mersenne primes (with distribution such that the number of such primes below exponent $n$ is approximated by $\exp(\gamma) \cdot \log(\log(n)))$ for Euler-Mascheroni constant $\gamma$), it follows from Theorem \ref{T_3_4} that $\limsup_{c \ra \infty} s_c = \infty$ \cite{pomerance1981}. Conversely, if $s_c$ is bounded above, then it follows that the conjecture that there exist infinitely many Mersenne primes fails.

These results motivate the following analogue of the $abc$-conjecture for the DGM quality; although this analogue is not quite as interesting as the conjecture itself, since we establish that the quality is \textit{unbounded} asymptotically, it nonetheless lends itself to modifications that will be of interest in Section \ref{S_DGM_Class}.

\begin{theorem}[DGM Quality Asymptotics]\label{C_abc_DGM}
We have that $\limsup_{c\to\infty} s_c = \infty$.
\end{theorem}
We remark that Theorem \ref{C_abc_DGM} follows from either (i) the infinitude of Mersenne primes, (ii) the infinitude of Fermat primes, (iii) the infinite of primes of the form $p^i+q^j$ for fixed primes $p,q$, or (iv) any case covered by Theorem \ref{T_main}. These observations follow from Theorems \ref{T_2_3}, \ref{T_power_of_p_q}, \ref{T_main}, and \ref{T_3_4}.

\subsubsection{Alternative Conditional Proof through Twin Primes}

The constructions in Theorems \ref{T_2_3}-\ref{T_3_4}  all share the structural property that at most one of the entries of the $abc$-triple $(a, b, c)$ remains a prime, such that $\omega(abc)$ remains invariant over the resulting families. We now give a distinct conditional route in which two such primes grow with the size of $c$, such that the constraint is provided by twin primes. This yields a new family of triples whose existence is not subsumed by Theorems \ref{T_2_3}–\ref{T_3_4}, since two entries of the triple vary (rather than being drawn from a fixed prime set).\footnote{While this method is less computationally efficient, it motivates the result in section \ref{S_chen} that gives an unconditional proof of Theorem \ref{C_abc_DGM}.}

\begin{theorem}[Twin-prime approach to Theorem \ref{C_abc_DGM}]\label{T_twin_DGM}
All triples of the form $(a,b,c)=(2,p,p+2)$ for $p$ prime satisfy
\[
q_{\mathrm{DGM}}(2,p,p+2) \in \mathcal{O} \left( \frac{(\ln c)^{1/3}}{3\,(\ln 2)^{1/3}} \right)
\]
as $c \ra \infty$, and hence if the twin prime conjecture holds, $\limsup_{c\to\infty} s_c=\infty$, establishing Theorem \ref{C_abc_DGM}.
\end{theorem}
\begin{proof}
Accordingly, all such triples $(a, b, c) = (2, p, p + 2)$ are $abc$-triples, since $p$ is odd and both $p, p + 2$ are primes. Thus, $\omega = 3$ results, with all three primes being distinct since $p\ge 3$ is required of a twin prime pair.

Thus, applying equation (\ref{eq3}), we have
\[
q_{\mathrm{DGM}}(2,p,p+2)=\frac{\ln(p+2)}{3 \left(\ln 2\cdot\ln(p)\cdot \ln(p+2) \right)^{1/3}}.
\]
For large $p$, we may invoke the approximation that $\ln(p+2)=\ln p+\ln(1+2/p)=\ln p+\mathcal{O}(1/p)$, and hence $\ln p\cdot\ln(p+2)=(\ln p)^2(1+o(1))$, yielding the result that
\[
({\ln 2\cdot \ln p\cdot \ln(p+2)})^{1/3} = (\ln 2)^{1/3}(\ln p)^{2/3}(1+o(1)).
\]
Substitution yields
\begin{align*}
q_{\mathrm{DGM}}(a, b, c) &= \frac{\ln(c)}{3\,(\ln 2)^{1/3}(\ln(c))^{2/3} (1+o(1))} \\& \in \mathcal{O}\left( \frac{(\ln(c))^{1/3}}{3\,(\ln 2)^{1/3}} \right)
\end{align*}
for triples of this form. Therefore, up to constant scaling, the DGM quality grows proportionately to $(\ln(c))^{1/3}$ along this family; thus, if the twin-prime family is infinite, then there are infinite such triples of the form $(2, p, p + 2)$, and we have $\limsup_{c \ra \infty} s_c=\infty$.
\end{proof}

We remark that Theorem \ref{T_twin_DGM} is structurally distinct from the strategy of Mersenne primes or fixed families of primes, since the twin prime direction yields a quality metric that grows slightly \textit{slower} in $\ln(c)$ asymptotically, and harnesses an entirely disjoint family of triples, yet nonetheless gives a divergent result for $s_c$.\footnote{By the Hardy-Littlewood prime $k$-tuples heuristic (which remains unproven), the density of twin primes below $N$ asymptotically approaches $2 C_2 N/(\ln N)^2$ with $C_2$ the twin-prime constant, so this family is heuristically far denser than the Mersenne family, which may be advantageous for future computational methods on the DGM quality \cite{hardy1923some}.}

\subsubsection{Implications of Chen's Theorem}
\label{S_chen}
All preceding routes to establish Theorem \ref{C_abc_DGM} are conditional on conjectures (for instance, the Twin-Prime or Lenstra-Pomerance-Wagstaff conjectures). We now demonstrate that, upon coupling the twin-prime construction provided by Theorem \ref{T_twin_DGM} with a slight weakening involving numbers that are the products of exactly two primes ($P_2$ numbers), then the theorem holds unconditionally. To establish this, we invoke the following result \cite{chen1973}.

\begin{theorem}[Chen \cite{chen1973}]\label{T_Chen}
There exist infinitely many primes $p$ such that $p+2$ is either (i) prime, or (ii) a product of two primes – i.e., a $P_2$ number.
\end{theorem}

Let the set $S$ be the set of primes $p \in \Z^{+}$ such that $p + 2$ lies in $\mathbb{P} \cup P_2$. Then, by the above theorem, we know $|S|=\infty$. Harnessing this result gives the following.

\begin{theorem}[Unconditional proof of Theorem \ref{C_abc_DGM}]\label{T_Chen_DGM}
We have $\limsup_{c \ra \infty} s_c = \infty$ per the following family: for all $p\in S$, the triple $(a, b, c) = (2,p,p+2)$ satisfies
\[
q_{\mathrm{DGM}}(a, b, c) \;\ge\; \frac{(\ln(c))^{1/4}}{4\,(\ln (2)/4)^{1/4}}.
\]
\end{theorem}
\begin{proof}
Fix some $p\in S$ with $p\ge 3$; Theorem \ref{T_Chen} guarantees infinitely many such $p$. Note that either $p + 2$ lies in $\mathbb{P}$, or it is a $P_2$ number. In either case, since $p, p + 2$ are odd, the elements $2, p, p + 2$ are pairwise coprime, such that $(a, b, c) := (2, p, p + 2)$ is an $abc$-triple. We now split into two cases. 

The first case is when $p + 2$ is prime. In this case, we simply apply the argument of Theorem \ref{T_twin_DGM} since $\omega = 3$ in this case, we have
\[
q_{\mathrm{DGM}}(a, b, c) \ge \frac{(\ln c)^{1/3}}{3\,(\ln(2))^{1/3}} \ge \frac{\ln(c)^{1/4}}{4 (\ln(2)/4)^{1/4}},
\]
where the latter inequality holds for sufficiently large $c$.

In the latter case, we have that $p + 2 \in P_2$, and hence $p + 2 = qr$ for $q, r$ prime and, without loss of generality, $q \leq r$. Note that since $p+2$ is odd, both $q, r$ are odd; additionally, we have $p \nmid p + 2$, since $p \ge 3$, and hence $p$ does not lie in the set $\{q, r\}$. Therefore, the primes dividing $abc=2pqr$ are exactly $\{2,p,q,r\}$, which give $\omega = 4$. We note that the subcase $q = r$ collapses to $\omega = 3$ and gives an even stronger bound than the twin prime case above, so we assume $q < r$ for the remainder of the derivation. Applying (\ref{eq3}) gives
\[
q_{\mathrm{DGM}}(a, b, c) = \frac{\ln c}{4({\ln 2\cdot\ln p\cdot\ln q\cdot\ln r})^{1/4}}.
\]
Noting that $\ln(q) + \ln(r) = \ln(qr) = \ln(c)$, AM-GM applied to $\ln(q), \ln(r)$ gives $\sqrt{\ln(q) \ln(r)} \leq \ln(c)/2$. Hence, together with $p < c$, we have
\[
\ln (2)\cdot\ln (p)\cdot\ln (q)\cdot\ln (r) \;\le\; \frac{\ln (2)}{4}\,(\ln (c))^3.
\]
Therefore, taking fourth roots,
\[
q_{\mathrm{DGM}}(a,b,c) \;\ge\; \frac{\ln (c)}{4\,(\ln(2)/4)^{1/4}(\ln (c))^{3/4}} \;=\; \frac{(\ln (c))^{1/4}}{4\,(\ln (2)/4)^{1/4}},
\]
and the desired inequality holds in either case. Since $|S|=\infty$, $c=p+2$ is unbounded along such triples corresponding to $p\in S$, so the right-hand side tends to $\infty$. Thus, $\limsup_{c\to\infty}s_c=\infty$ and Theorem \ref{C_abc_DGM} holds.
\end{proof}
We remark briefly on the exponent in this proof. Our exponent on the logarithmic term is $1/4$, which derives from the AM-GM usage; however, the proof of Chen's theorem implies that the small prime factor $q$ in the proof above satisfies that $\log(q)/\log(p + 2)$ is bounded above by $1/2 + o(1)$.\footnote{This follows from asymptotics under the requirement $p^{1/10} < q \leq p^{1/2} \leq r \leq p^{0.9}$ \cite{chen1973}.} Under such size control of $q$, the bound improves to $q_{\text{DGM}} (a, b, c) \sim (\ln(c))^{1/3 - \eps}$ along this family, which matches those of the twin prime family. We leave a fully unconditional refinement of the critical exponent of divergence of the DGM quality as an open problem.\footnote{One might define a \textit{running average} $r(c_0) := \sum_{x \in A} q_{DGM}(x)/|A|$, where $A = \{\text{triples\ } (a, b, c): c \leq c_0\}$; this experimentally tends to a constant near $0.528$ as $c \ra \infty$.}

\subsection{Selected Experimental Observations}

Figure \ref{fig:comparison} demonstrates the evaluation of $q_{\mathrm{DGM}}$ on the classes of triples in the previous section, with the highest value being $4543.95$ from the Mersenne family. Table \ref{Table_2} lists all triples $(1,b,c)$ with $c\le 140{,}000$ and $q_{\mathrm{DGM}}\ge 1.6$; we note that the maximum $\omega$ is $4$, such that triples with small overall numbers of prime divisors are prioritized. Table \ref{T_comparing_abc_to_dgm}, comparatively, evaluates the known triples corresponding to $\arg \max_{(a, b, c)} q_s(a, b, c)$ for the standard quality under $q_{\mathrm{DGM}}$: all are DGM-hits, but \emph{none} are of extremely high DGM quality.

\begin{figure}[!t]
    \centering
    \includegraphics[width=0.6\linewidth]{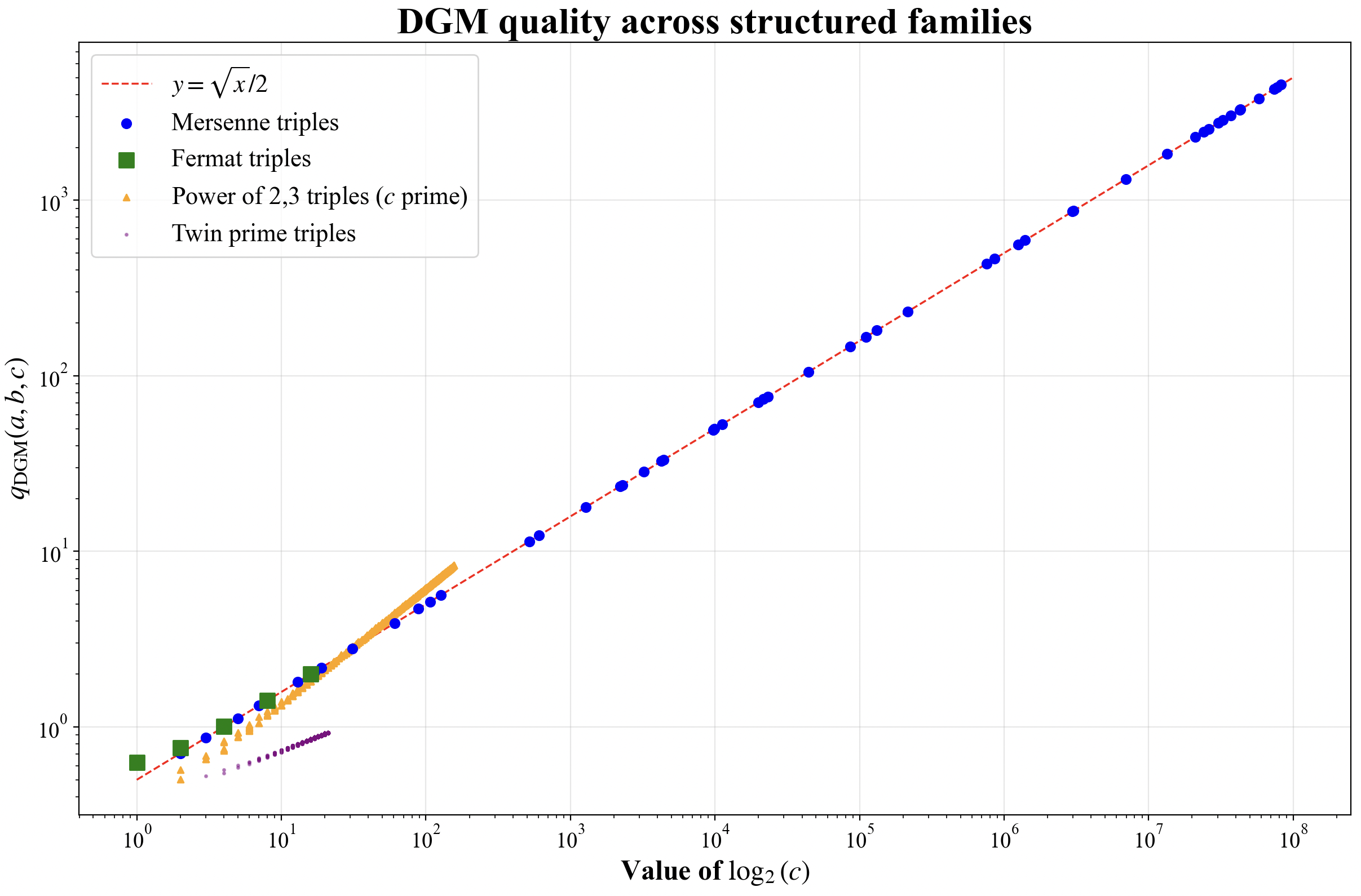}
    \caption{Comparative attainment of high-quality triples through distinct methods, enumerated in Section \ref{S_families}. The Mersenne method yields significantly higher quality triples due to \textit{known} enumerated Mersenne primes.}
    \label{fig:comparison}
\end{figure}

\begin{table}[t]
\centering\small
\begin{tabular}{|c|c|c|c|c|}
\hline
$a$ & $b$ & $c$ & $\omega$ & $q_{\mathrm{DGM}}$ \\
\hline
1 & 4374   & 4375   & 4 & 1.68658 \\
\textcolor{blue}{1} & \textcolor{blue}{8191}   & \textcolor{blue}{8192}   & \textcolor{blue}{2} & \textcolor{blue}{1.80279} \\
1 & 13121  & 13122  & 3 & 1.63528 \\
1 & 23327  & 23328  & 3 & 1.70077 \\
1 & 52488  & 52489  & 3 & 1.79101 \\
1 & 59048  & 59049  & 4 & 1.6593  \\
1 & 62207  & 62208  & 3 & 1.80963 \\
\textcolor{red}{1} & \textcolor{red}{65536}  & \textcolor{red}{65537}  & \textcolor{red}{2} & \textcolor{red}{2 (rd.)} \\
1 & 73727  & 73728  & 3 & 1.82815 \\
\textcolor{blue}{1} & \textcolor{blue}{131071}   & \textcolor{blue}{131072}   & \textcolor{blue}{2} & \textcolor{blue}{2.06155} \\
1 & 139967 & 139968 & 3 & 1.8972  \\
1 & 139968 & 139969 & 3 & 1.8972 \\
\hline
\end{tabular}
\caption{Selected triples with $q_{\mathrm{DGM}}\ge 1.6$, $a=1$, $c\le 140{,}000$. Blue triples correspond to Mersenne triples, and the red triple is the largest known Fermat triple.}
\label{Table_2}
\end{table}

\begin{table*}[t]
    \centering
    \begin{tabular}{|c|c|c|c|c|c|}
        \hline
        $a$ & $b$ & $c$ & \textbf{Standard Quality} & \textbf{DGM Quality} & $\omega$ \\
        \hline
        $2$ & $3^{10} \cdot 109$ & $23^5$ & $1.6299$ & $2.1424$ & $4$ \\
        $11^2$ & $3^2 \cdot 5^6 \cdot 7^3$ & $2^{21} \cdot 23$ & $1.6260$ & $1.8226$ & $6$ \\
        $19 \cdot 1307$   & $7 \cdot 29^2 \cdot 31^8$ & $2^8 \cdot 3^{22} \cdot 5^4$ & $1.6235$ & $2.0388$ & $8$ \\
        $283$ & $5^{11} \cdot 13^{2}$ & $2^8 \cdot 3^8 \cdot 17^3$ & $1.5808$ & $1.9809$ & $6$ \\
        $1$ & $2 \cdot 3^7$ & $5^4 \cdot 7$ & $1.5679$ & $1.6866$ & $4$ \\
        $7^3$ & $3^{10}$ & $2^{11} \cdot 29$ & $1.5471$ & $1.8386$ & $4$ \\
        $7^2 \cdot 41^2 \cdot 311^3$ & $11^{16}\cdot 13^2 \cdot 79$ & $2\cdot3^3\cdot5^{23}\cdot953$ & $1.5444$ & $1.9180$ & $10$\\
        \hline
    \end{tabular}
    \caption{Best $q_s$ triples under $q_{\mathrm{DGM}}$. The seventh-highest $q_s$ triple has surprisingly low $q_{\mathrm{DGM}}$ due to large $\omega$.}
    \label{T_comparing_abc_to_dgm}
\end{table*}

\vspace{-3mm}

\section{DGM Quality Class}
\label{S_DGM_Class}

\subsection{Motivation and Definition}

Noting that the DGM Quality asymptotically tends to infinity, while the standard quality conjecturally tends to $1$ in a $\lim\sup$ fashion, we produce tunable variants of the DGM quality using parameters $\alpha, \beta$, as follows.

\begin{definition}[DGM Quality Class]\label{D_DGM_class}
For specific $\alpha\ge 0$, $\beta>0$ and $U=\{p_i^{\omega^\alpha}:p_i \mid abc, p_i \in \mathbb{P}\}$, we set
\[
q_C(a,b,c;\alpha,\beta) := \frac{\ln(c)}{\bigl(\ln(\mathrm{DGM}(U))\bigr)^\beta}.
\]
\end{definition}
Note that $q_C(\cdot;1,1)$ reduces to $q_{\mathrm{DGM}}$, and hence $\alpha$ corresponds to a finer influence of $\omega$ itself, whereas $\beta$ controls overall growth of the denominator.

Fix $\alpha, \beta$ as above. By noting that
\begin{align*}
    \ln(\mathrm{DGM}(U)) &= \exp\bigl(\tfrac{1}{\omega}\textstyle\sum_i\ln\ln(p_i^{\omega^\alpha})\bigr)\\
&= \exp\bigl(\alpha\ln\omega+\tfrac{1}{\omega}\textstyle\sum_i\ln\ln p_i\bigr)\\
&= \omega^\alpha\cdot\sqrt[\omega]{\textstyle\prod_i\ln p_i},
\end{align*}
it follows that 
\begin{align}
\label{eq_primefact}
    q_C(a,b,c;\alpha,\beta) = \frac{\ln c}{\bigl(\omega^\alpha\cdot\sqrt[\omega]{\prod_{i=1}^\omega \ln p_i}\bigr)^\beta}
\end{align}
over all $abc$-triples ($a, b, c$), per a similar rewriting of our DGM quality. Therefore, an analogous result to Theorem \ref{T_13} may be established.

\begin{theorem}\label{T_6.3}
For all $abc$-triples and $\alpha \in [0,1]$, $\beta>0$,
\[
q_C(a,b,c;\alpha,\beta) > \frac{\ln(c)}{(\omega^{\alpha - 1} \ln(abc))^{\beta}} \ge \frac{1}{3(\omega^{\alpha-1})^\beta}>\frac
{1}{3}.
\]
\end{theorem}
\begin{proof}
We apply AM-GM, such that $\sqrt[\omega]{\prod\ln p_i}\le\tfrac{1}{\omega}\sum\ln p_i$; Lemma \ref{L_ln_primes_DGM} gives that $q_C(a, b, c; \alpha, \beta) > \ln(c)/((\omega^{\alpha - 1} \ln(abc))^{\beta})$. This bound holds independent of the restriction on $\alpha$. Additionally, substituting $\ln(abc) < 3 \ln(c)$ and $\omega^{\alpha - 1} \leq 1$ for $\alpha \leq 1$, we have the desired inequality.
\end{proof}

We may also establish a few elementary results with respect to the variation of $\alpha$ while $\beta$ is fixed. Indeed, for $0 \leq \alpha_1 < \alpha_2$ and some fixed $\beta > 0$, we have that $q_C(\cdot; \alpha_1, \beta) > q_C(\cdot; \alpha_2, \beta)$, since the denominator of equation (\ref{eq_primefact}) is smaller for $\alpha_1$ while $\beta$ is held constant. Indeed, for the extreme case $\alpha = 0$, we observe $q_C(a,b,c;0,1)=\omega\cdot q_{\mathrm{DGM}}$. In this case, we observe that the quality $q_C(a,b,c;0,1)$ diverges (as a consequence of Theorem \ref{C_abc_DGM}) and is bounded below by $1$ (independently resolving the cases $\omega \ge 3$, where we apply Theorem \ref{T_13}; $\omega = 2$, where we note that Theorem \ref{T_6.3} with $\alpha = 0$ yields $q_C(a, b, c;0, 1) > \omega/2 = 1$; and $\omega = 1$, which gives a single triple). Additionally, a natural extension of Theorem \ref{T_3_4} gives $q_C(a,b,c;0,1)\ge\sqrt{n}$ for Mersenne triples, and similarly for triples of the form in Theorem \ref{T_Chen_DGM}, such that divergence is guaranteed.

We may also consider the variation of $\beta$ under fixed values of $\alpha$. We note that the increase of $\beta$ slows growth of the quality metric monotonically. That is, $\beta=2$ keeps $q_C(\cdot;1,2)\le 2$ for all examined triples; the high-quality triples for $\beta=2$ thereby form a refined subset of those for $\beta=1$. Thus, the asymptotic control of $\beta$ lends itself to the following conjecture.

\begin{conjecture}\label{Smooth_class_ABC_Analogue}
For every $\alpha\ge 0$, there exists $\beta=\beta(\alpha)>0$ such that
\[
\limsup_{c\to\infty}\max\{q_C(a,b,c;\alpha,\beta):(a,b,c)\text{ abc-triple}\}=L
\]
for some finite, nonzero $L$.
\end{conjecture}

\subsection{Phase Transitions in DGM Parameter Regimes}
\label{S_pt}
Conjecture \ref{Smooth_class_ABC_Analogue} asserts that, for each $\alpha$, there exists some $\beta$ giving a finite, nonzero limit. We investigate partial justifications of this claim through variants with respect to specific values of $\beta$ that give \textit{phase transitions} in the asymptotic behavior of the DGM quality class. We harness explicit families from Section \ref{S_families}, identifying a \textit{critical} $\beta$ from which the asymptotic quality transitions from divergence to convergence to a finite value.

Let $\mathcal{F}$ be a family of $abc$-triples such that $\omega(abc)=\omega$ is constant (motivated by Section \ref{S_DGM_Class}), and $k$ of the $\omega$ primes used grow on the order of $\Theta(c)$, such that the other $\omega-k$ primes are uniformly bounded. Call this the $(\omega,k)$-\textit{condition}. For instance, Mersenne triples have $(\omega,k)=(2,1)$; twin-prime triples $(2,p,p+2)$ satisfy $(\omega, k) = (3,2)$; and power of $2,3$ triples (with $c$ prime) have $(\omega, k) = (3,1)$. We utilize this principle to establish the following.

\begin{theorem}[Phase transition for $q_C$ Class]\label{T_phase_transition}
Suppose there exists an infinite family $\mathcal{F}$ of $abc$-triples satisfying the $(\omega,k)$-condition, for fixed $\omega, k$. Then, it follows that the critical exponent is $\beta_0=\omega/k$, such that
\[
\lim_{c\to\infty,\,(a,b,c)\in\mathcal{F}} q_C \;=\; \begin{cases} \infty & \beta<\omega/k,\\[2pt] \dfrac{1}{\omega^{\alpha\omega/k} \cdot K^{1/k}} & \beta=\omega/k,\\[6pt] 0 & \beta>\omega/k, \end{cases}
\]
for a finite constant $K$.
\end{theorem}
\begin{proof}
Let $(a, b, c) \in \mathcal{F}$, and write the primes dividing $abc$ as $\{p_1,\dots,p_{\omega-k},q_1,\dots,q_k\}$, such that the primes $p_i$ are all bounded and we have that $\ln(q_j) = \ln (c) \cdot (1 + o(1))$ for each $j \in \{1, \dotsc, k\}$. Set $K := \prod_{i=1}^{\omega-k}\ln p_i $, which is constant with respect to $(a, b, c)$. Applying (\ref{eq_primefact}) gives
\[
q_C(a, b, c; \alpha, \beta) = \frac{\ln (c)}{\bigl(\omega^\alpha\sqrt[\omega]{\prod_i\ln (p_i)\cdot\prod_j\ln (q_j)}\bigr)^\beta}.
\]
Since $\ln q_j=\ln c\cdot(1+o(1))$ for $j=1,\dots,k$, the inner product factors as $\textstyle\prod_{i}\ln(p_i) \prod_j \ln(q_j) = K\cdot(\ln (c))^k\cdot(1+o(1))$. Substitution gives
\begin{align*}
q_C(a, b, c; \alpha, \beta) &= \frac{\ln (c)}{\omega^{\alpha\beta}\,\left(\prod_{i=1}^{\omega-k}\ln p_i \right)^{\beta/\omega}(\ln (c))^{k\beta/\omega} (1+o(1))}\\
 &= \omega^{-\alpha\beta}K^{-\beta/\omega}(\ln (c))^{1-k\beta/\omega}(1+o(1)).
\end{align*}
Noting that all terms but the $\ln(c)$ term are constant in the resulting expression, it follows that $\mathrm{sgn}(1-k\beta/\omega)$ determines the asymptotic behavior of $q_C(\cdot; \alpha, \beta)$ on such triples. Indeed, at $\beta=\omega/k$, $\ln(c)$ vanishes and the limit equals $\omega^{-\alpha\omega/k}K^{-1/k}$, as desired. Note that, assuming this result, the resulting $\lim\sup q_C(a, b, c_0)$ over triples $(a, b, c_0)$, as $c_0 \ra \infty$, must tend to $\infty$ for $\beta < \omega/k$ and must be bounded below by $\omega^{- \alpha \omega/k} K ^{-1/k}$ at $\beta = \omega/k$.
\end{proof}

\begin{corollary}\label{C_4_5}
At the critical $\beta_0=\omega/k$, the (conditional) limit values along specified families are as follows:
\end{corollary}
\begin{center}\small
\begin{tabular}{|l|c|c|c|c|}
\hline
Family & $\omega$ & $k$ & $\beta_c$ & $\lim_{c} q_C$ at $\beta_0$ \\
\hline
Mersenne Triple & 2 & 1 & $2$ & $1/(2^{2\alpha}\ln(2))$ \\
Power of $2,3$ Triples & 3 & 1 & $3$ & $1/(3^{3\alpha}\ln(2) \ln(3))$ \\
Twin primes $(2,p,p\!+\!2)$ & 3 & 2 & $3/2$ & $1/(3^{3\alpha/2}(\ln(2))^{1/2})$ \\
\hline
\end{tabular}
\end{center}

The above corollary follows from direct substitution into Theorem \ref{T_phase_transition} under appropriate choices of $K$. Thus, the DGM Quality class  provides a parametrized tuning with respect to the families in Section \ref{S_DGM_Quality}, such that for suitable $\beta$, the analogue of Conjecture \ref{Smooth_class_ABC_Analogue} appears to hold. The significance of Theorem \ref{T_phase_transition} is therefore as follows: the critical $\beta_0$ effectively gives us a \textit{saturation condition} for the relevant families of triples, comparing the number of primes \textit{used} for the family $(\omega$) to the number that grow ($k$). As $\beta$ is increased, families progressively saturate in a manner of pruning the highest-quality triples; indeed, crossing $\beta \ge \beta_0 = \omega/k$ for the corresponding family ensures that such a family is not considered within the corresponding $\limsup_{c \ra \infty} q_C(\cdot, c; \cdot,  \beta)$ in Conjecture \ref{Smooth_class_ABC_Analogue}. 

Although this result does not establish Conjecture \ref{Smooth_class_ABC_Analogue}, it provides a lower bound on $s_c$ for each such family; extending this result to a complete proof would require control of $\omega(abc)$ for all triples, perhaps through Evertse's S-unit equation theorem \cite{granville2002}. In the restricted case of $\omega = 2$, however, an upper bound may be demonstrated. We now establish this $\limsup$ result in this case for the DGM Class, invoking Milailescu's Theorem, as follows.

\begin{theorem}[$\omega=2$ Situation]\label{T_omega2_trichotomy}
For $\alpha\ge 0$, let $s_c^{(2)}(\alpha,\beta) := \sup\{q_C(a,b,c;\alpha,\beta):\omega(abc)=2\};$ then,
\[
\limsup_{c\to\infty} s_c^{(2)}(\alpha,\beta) \;=\;
\begin{cases}
+\infty & \beta<2\\
\frac{1}{2^{2\alpha}\,\ln 2} & \beta=2\\
0 & \beta>2;
\end{cases}
\]
the cases are conditional on the infinitude of Mersenne primes (or, equivalently, of Fermat primes).   
\end{theorem}

\begin{proof}
We first discuss the structure of triples with $\omega = 2$. If $(a, b, c)$ is such a triple, then let the two primes be of the form $\{p_1,p_2\}$, $p_1 \neq p_2$. Since $a,b$ are coprime, we may write $a=p_1^i$ and $b=p_2^j$ for some $i,j\ge 0$ up to relabeling. Hence, we have $c=p_1^i+p_2^j$, which must be of the form $c=p_1^m p_2^n$ for some $m,n\ge 0$. Additionally, the coprime constraint guarantees that either $c = 1$, which is impossible (as $a, b \ge 1$), or that one of $a, b$ must be $1$. In the latter case, assume that $a = 1$, forcing the equation $1 + p_2^j = p_1^m p_2^n$. Thus, either $j = 0$, giving the triple $(1, 1, 2)$, or $n = 0$ by coprimality. In the $n = 0$ situation, we have $b=p_1^m-1=p_2^j$, such that $p_2^j+1=p_1^m$. By Mihailescu's theorem, the only solution to this equation with indices $m,j\ge 2$ is given by $(p_1, m, p_2, j) = (2, 3, 3, 2)$, corresponding to $(a, b, c) = (1, 8, 9)$ \cite{mihailescu2004primary}. The case $j = 1$ gives Mersenne triples by parity considerations, such that $p_1 = 2$ and $p_2 = 2^m - 1$. Moreover, the case $m = 1$ gives Fermat triples by parity on $p_2$, such that $p_2 = 2$ and $p_1 = 2^j + 1$. The situation where $b = 1$ gives identical triples up to symmetry in $a, b$. Hence, the full class of triples of this form are Fermat triples and Mersenne triples, with the exception of $(1, 1, 2)$.

Therefore, an infinite sequence of $abc$-triples with $\omega = 2$ and $c \ra \infty$ must contain an infinite subsequence drawn from either Mersenne triples or Fermat triples. In either case, the primes dividing $abc$ lie in $\{2,p\}$ where $p \sim c$ (up to $\pm 1$). Thus, by equation (\ref{eq_primefact}),
\begin{align*}
q_C (a, b, c) &= \frac{\ln(c)}{\bigl(2^\alpha\sqrt{\ln (2)\cdot\ln (p)}\bigr)^\beta} =  \frac{(\ln (c))^{1-\beta/2}}{2^{\alpha\beta}(\ln (2))^{\beta/2} (1+o(1))},
\end{align*}
by $\ln(p)=\ln(c)(1 + o(1))$. Thus, we have a trichotomy dependent on $\sgn(1 - \beta/2)$: if $\beta < 2$, then $q_C(a, b, c; \alpha, \beta) \ra \infty$ along this family, if it is indeed infinite; if $\beta = 2$, then $q_C \ra 1/(2^{2\alpha}\ln 2)$ along this family; and if $\beta > 2$, then $q_C(a, b, c; \alpha, \beta) \ra_{c \ra \infty} 0$. All three results are conditional on infinitude of this family. Moreover, in the case $\beta < 2$, $\limsup_{c \ra \infty} s_c \ge \limsup_{c \ra \infty} s_c^{(2)}$ forces that $\limsup_{c \ra \infty} s_c = \infty$, conditionally, in this case.
\end{proof}

We note that $\beta_0=2$ in the DGM Class thereby behaves in a functionally analogous manner to the critical value of $\eps \ra 0$ in the standard formulation of the conjecture (Conjecture \ref{C_abc_conjecture}), such that the statement that finitely many triples exceed the critical value corresponds to the conjecture itself.

The $\omega$-restriction in Theorem \ref{T_omega2_trichotomy} is essential; in particular, a bound that is uniform in \textit{all} $\omega$ on a quality that is functionally similar to the standard quality is essentially equivalent to the conjecture itself \cite{waldschmidt2015,granville2002}. However, to make progress on \textit{unconditionally} establishing boundedness claims, we appeal to the results established through Chen's Theorem and in Theorem \ref{T_Chen_DGM}, as well as boundedness conditions that invoke $abc$ itself.\footnote{We remark that Vojta's conjecture would settle the uniform case unconditionally if established \cite{vojta1987}.}
\begin{theorem}[Unconditional lower bound]\label{T_uncond_LB}
For every $\alpha\ge 0$ and all $\beta<4/3$,
\[
\limsup_{c\to\infty} s_c = \infty,
\]
\end{theorem}
where $s_{c_0}=\sup_{\text{$abc$ triples} (a,b,c_0)}\{q_C(a,b,c;\alpha,\beta)\}$.

\begin{proof}
We apply the second family used in the proof of Theorem \ref{T_Chen_DGM}, which is an infinite family of triples $(2, p, qr)$ with $\omega = 4$, such that three of the primes used grow as $\Theta(c)$, $\Theta(\sqrt{c}),\Theta(\sqrt{c})$. Applying the bound in this theorem for such triples and $\omega = 4$, we have
\[
q_C(a, b, c;\alpha,\beta) \;\ge\; \frac{(\ln (c))^{1-3\beta/4}}{4^{\alpha\beta}\,(\ln (2)/4)^{\beta/4}}.
\]
For $\beta<4/3$, we have $1-3\beta/4>0$, so the $\frac{(\ln (c))^{1-3\beta/4}}{4^{\alpha\beta}\,(\ln (2)/4)^{\beta/4}} \ra \infty$ as $c=p+2\to\infty$ along the infinite set $S$ from Chen's Theorem.
\end{proof}

We remark that, under the twin prime conjecture, the threshold of Theorem \ref{T_uncond_LB} tightens to $\beta < 3/2$. We now invoke a conditional upper bound on $q_C$ that derives from the $abc$-conjecture itself.

\begin{theorem}[Upper bound, assuming $abc$]\label{T_cond_UB}
Assume Conjecture \ref{C_abc_conjecture}, and fix $\omega_0\ge 2$. Then, for all $\varepsilon>0$, $\alpha\ge 0$, $\beta\ge\omega_0$, and $abc$-triples with $\omega(abc)=\omega_0$,
\[
\limsup_{c\to\infty,\,\omega=\omega_0} q_C(a,b,c;\alpha,\beta)
 \le (1+\varepsilon)^{\beta/\omega_0}\,\omega_0^{\beta/\omega_0-\alpha\beta} K(\beta, \omega_0),
\]
where $K(\beta, \omega_0) := (\ln (2))^{-\beta(1-1/\omega_0)}$. Thus, $\limsup_{c \ra \infty} q_C(a, b, c; \alpha, \beta)  < \infty$ for $\beta = \omega_0$ and vanishes for all $\beta>\omega_0$.
\end{theorem}
%No corresponding uniform-in-$\omega$ statement is available: along Mersenne triples (also conditional on LPW) $q_{\mathrm{DGM}}\to\infty$, demonstrating that abc alone does not bound $q_C$ uniformly in $\omega$.

\begin{proof} Let $(a, b, c)$ satisfy $\omega(abc) = \omega_0$. Let $p_1,\dots,p_{\omega_0}$ be the prime divisors of $abc$, with $p_{m} := \max \{p_i\}_{i = 1}^{\omega_0}$. Since $\mathrm{rad}(abc)=\prod_{i =1}^{\omega_0} p_i\le p_{m}^{\omega_0}$, we have $\ln(p_m) \ge \frac{\ln(\mathrm{rad}(abc))}{\omega_0}$.
Let $\eps > 0$. The $abc$-conjecture gives $\ln c\le(1+\varepsilon)\ln(\mathrm{rad}(abc))$ for $c$ greater than some $c_0$; hence, $\ln p_{m}\ge\ln c/((1+\varepsilon)\omega_0)$ for such $ c> c_0$. Combining this result with $\ln (p_i) \ge\ln (2)$ for the remaining primes, we have
\[
\textstyle\prod_{i = 1}^{\omega_0} \ln (p_i) \;\ge\;(\ln (2))^{\omega_0-1}\cdot\frac{\ln (c)}{(1+\varepsilon)\omega_0},
\]
such that if we let $G = \left(\textstyle\prod_{i = 1}^{\omega_0} \ln (p_i)  \right)^{1/\omega}$, it follows that
\[
G\;\ge\;(\ln 2)^{1-1/\omega_0}\cdot\bigl((1+\varepsilon)\omega_0\bigr)^{-1/\omega_0}(\ln c)^{1/\omega_0}.
\]
Substituting into $q_C(a, b, c; \alpha, \beta)=\ln (c)/(\omega_0^\alpha \cdot G)^\beta$,
\[
q_C(a, b, c; \alpha, \beta) \le \frac{(1+\varepsilon)^{\beta/\omega_0}\,\omega_0^{\beta/\omega_0-\alpha\beta}}{(\ln (2))^{\beta(1-1/\omega_0)}}\,(\ln (c))^{\,1-\beta/\omega_0}
\]
Hence, for $\beta = \omega_0$, the exponent $1 - \beta /\omega_0 \leq 0$, such that the right-hand side is bounded; moreover, the right-hand bound vanishes when $\beta > \omega_0$.
\end{proof}

Theorems \ref{T_omega2_trichotomy}, \ref{T_uncond_LB}, and \ref{T_cond_UB} give the following phase diagram with respect to the value of $\beta$:

\begin{center}\small
\begin{tabular}{|l|l|}
\hline
Region & $\limsup_{c \ra \infty} s_c$ \\
\hline
$\beta<4/3$ & $\infty$ unconditionally (Thm \ref{T_uncond_LB}) \\
$4/3\le\beta<2$ & $\infty$ if $\exists$ inf. Mersenne primes \\
$\beta= \omega_0$, for $\omega(abc)=\omega_0$ & finite, given \textit{abc} (Thm \ref{T_cond_UB})\\
$\beta>\omega_0$, for $\omega(abc)=\omega_0$ & $0$, given \textit{abc} (Thm \ref{T_cond_UB})\\
\hline
\end{tabular}
\end{center}

We note that, even assuming the $abc$-conjecture, there does not exist a bound on $\limsup_{c \ra \infty} s_c$ that is uniform, since the critical $\beta_0=\omega/k$ for a family of triples $(a, b, c)$ with $k$ primes growing as $\Theta(c)$ primes grows linearly with $\omega$ and thereby unbounded as $\omega \ra \infty$. That is, a uniform bound on $\limsup_{c \ra \infty} q_C$ would be a statement strictly stronger than the $abc$-conjecture itself. 

An additional bound may be established via \textit{smoothness} considerations on the relevant $abc$-triples.

\begin{theorem}[Smooth-Prime Bound]\label{T_cond_UB_Baker}
    For fixed $\omega_0 \ge 3$, $A > 0$, let $\mathcal{T}_{\omega_0, A}$ comprise of $abc$-triples with $\omega(abc)=\omega_0$ and $P(abc) \le c^A$, where $P(n)$ is the largest prime dividing $n$. Then for all 
$\alpha \ge 0,\beta \ge \omega_0$,
\[
\limsup_{\substack{(a,b,c)\in\mathcal{T}_{\omega_0,A}, c\to\infty}} 
q_C(a,b,c;\alpha,\beta) \;<\; \infty.
\]
\end{theorem}

\begin{proof}
Let $(a,b,c)\in\mathcal{T}_{\omega_0,A}$; the prime factors dividing $abc$ lie in the set $S := \{p_1,\ldots,p_{\omega_0}\}$. 
Since $a+b=c$ with $\gcd(a,b)=1$, the pair $(u, v) = (a/c,b/c)$ satisfies an $S$-unit equation $u+v=1$, since $a/c, b/c$ are $S$-units. By a result due to Györy, there exists a constant $C(\omega_0)$, depending only on $\omega_0$, such that
\begin{equation}\label{eq5}
    \ln c \leq C(\omega_0)\prod_{i=1}^{\omega_0} \ln(p_i) =C(\omega_0)\,G^{\omega_0},
\end{equation}
where $G := \bigl(\prod_{i=1}^{\omega_0}\ln p_i\bigr)^{1/\omega_0}$ \cite{gyory1979explicit}. Applying $p_i \le P(abc) \le c^A$ for each $i$ gives $G \le A\ln c$. By (\ref{eq5}), we have $ \ln c \;\le\; C(\omega_0)\,(A\ln c)^{\omega_0} \;=\; C(\omega_0)\,A^{\omega_0}(\ln c)^{\omega_0}$, and hence $(\ln c)^{1-\omega_0} \le C(\omega_0) A^{\omega_0}$, or, equivalently,
\begin{equation}\label{eq_G_bound}
    G \;\le\; A\ln c \;\le\; A\cdot\bigl(C(\omega_0)A^{\omega_0}\bigr)^{\frac{1}{\omega_0-1}} 
    \;=:\; K(\omega_0, A),
\end{equation}
a constant depending only on $\omega_0$ and $A$. Substituting equation (\ref{eq5}) into (\ref{eq_primefact}) gives that
\begin{align*}
    q_C = \frac{\ln c}{\omega_0^{\alpha\beta}\,G^\beta} \leq 
    \frac{C(\omega_0)}{\omega_0^{\alpha\beta}}\,G^{\omega_0-\beta};
\end{align*}
since $G \le K(\omega_0, A)$ is a constant depending only on $\omega_0$ and $A$, which are both fixed with respect to triples, and 
$\omega_0 - \beta \le 0$ for $\beta \ge \omega_0$, the right side is bounded above by 
$C(\omega_0)\,\omega_0^{-\alpha\beta}\,K(\omega_0,A)^{\omega_0-\beta}$, 
which is finite. Hence, $\limsup_{c \ra \infty} q_C$ on this family is bounded above for all $\beta \ge \omega_0$, as desired.
\end{proof}

\subsection{Mapping Principles for Phase Transitions}

We now progress from phase transition results for constant $\omega$ to asymptotics of $q_C$ without restrictions on fixed $\omega$. As intuition, we define a map $\rho: (\Z^+)^3 \ra \mathcal{E}$ from triples to the set of semistable elliptic curves $\mathcal{E}$ by mapping each $abc$-triple to a corresponding semistable Frey elliptic curve, given in Weierstrass form \cite{silverman2009arithmetic, frey1986links}:
\[
     \rho(a, b, c) \mapsto E_{a, b, c}: y^2 = x(x - a)(x + b)
\]
Additionally, we use invariants of $E_{a,b,c}$ to evaluate  $q_C$, including the minimal discriminant $\Delta_E = 2^{-8} (abc)^2$ and the conductor $N_E$, which equals $\mathrm{rad}(abc) = \prod_{i=1}^{\omega} p_i$. 

Moreover, it must be noted that the Szpiro ratio $L = \ln|\Delta_E|/\ln N_E$ satisfies 
$L = \frac{2\ln(abc)+O(1)}{\ln N} \approx \frac{2\ln c}{\ln N} \;=\; 2 q_s(a, b, c)$ 
\cite{szpiro1981proprietes}. Szpiro's conjecture asserts $L \le 6+\varepsilon$ for all 
semistable elliptic curves over $\mathbb{Q}$, which would imply $q_s \lesssim 3$; more 
generally, bounding $L \le 2\sigma$ is equivalent to $\ln(c) \lesssim \sigma\ln(N)$ \cite{granville2002, szpiro1981proprietes}. 
This motivates the parametrized result below; we include a correction $-\lambda\omega\ln\omega$ 
that arises from the gap between $\ln(abc)$ and $\ln c$. That is, the bound on the Szpiro ratio $L \le 2\sigma$ gives $2\ln c \;\le\; 2\sigma\ln N - 2\ln(ab) + O(1)$. To estimate $\ln(ab)$, we use a heuristic given by the prime number theorem: if the $\omega$ primes dividing $abc$ are typical primes near their 
median size, $\ln N = \sum_{i=1}^{\omega}\ln p_i \sim  \omega\ln\!\left(\frac{\ln c}{\omega}\right) 
    \;=\; \omega\ln(\ln (c)) - \omega\ln(\omega)$; this gives $\ln(ab) \le \omega\ln P(abc) \sim \omega\ln\ln c$ and suggests the  correction is of order $\omega\ln\omega$. We thus impose the heuristic ansatz \[
    \ln (c) \;\le\; \sigma\ln (N) - \lambda\omega \ln(\omega),
\]
for positive $\lambda$, noting that this is implied by Szpiro's conjecture when $\sigma = 3$ and $\lambda$ is 
an absolute constant.\footnote{Note: the correction term is $o(\ln c_n)$ whenever 
$\omega_n \in o(\ln c_n / \ln(\ln (c_n)))$, which holds under the assumption $\omega_n \sim \delta(\ln c_n)^\kappa$ with $\kappa < 1$, as noted in the theorem itself.}

\begin{theorem}[Szpiro-Parametrized Phase Transition]
\label{T-szpiro}
Let $\mathcal{F} = \{(a_n, b_n, c_n)\}_{n=1}^{\infty}$ be a sequence of $abc$-triples with $c_n \to \infty$, and let $\omega_n = \omega(a_n b_n c_n)$. Assume that $\ln c_n \le \sigma \ln N_n - \lambda \omega_n \ln \omega_n$, for constants $\sigma \ge 1$ and $\lambda \ge 0$, with the Szpiro bound $L \approx 2\sigma$. If $\omega_n \sim \delta \cdot (\ln c_n)^\gamma$ for constants $\delta > 0$ and $0 < \gamma < 1$, then we have
\begin{align*}
    \limsup_{(a_n, b_n, c_n) \in \mathcal{F}, c_n \ra \infty} q_C = \begin{cases}
        \frac{1}{\delta^{1/\gamma} (\ln(2))^{\beta}} & \alpha \beta = 1/\gamma \\
        0 & \alpha \beta > 1/\gamma.
    \end{cases}
\end{align*}
\iffalse
\begin{enumerate}
    \item \textbf{Sub-Critical Regime ($\alpha\beta < \frac{1}{\gamma}$):} The dampening penalty is insufficient, yielding $\limsup_{n \to \infty} q_C = \infty$.
    \item \textbf{Critical Boundary ($\alpha\beta = \frac{1}{\gamma}$):} The logarithmic growth is perfectly neutralized, yielding a finite, non-zero geometric boundary constant:
    \[
        \limsup_{n \to \infty} q_C = \frac{1}{\delta^{1/\gamma} (\ln 2)^\beta}
    \]
    \item \textbf{Super-Critical Regime ($\alpha\beta > \frac{1}{\gamma}$):} The structural dampening dominates unconditionally, yielding $\lim_{n \to \infty} q_C = \frac{\ln c_n}{(\ln c_n)^{\gamma\alpha\beta}} = 0$.
\end{enumerate}
\fi
\end{theorem}

Before evaluating this asymptotic behavior, we establish a lower bound on the denominator of $q_C$, as follows.

\begin{lemma}
\label{sqfree-lemma}
Let $N = \prod_{i=1}^\omega p_i$ squarefree, with $\omega \ge 2$ distinct prime factors. Then, $G := \left(\prod_{i=1}^\omega \ln p_i\right)^{1/\omega}$ satisfies
\[
    G \ge (\ln (2))^{1-1/\omega} \left( \frac{\ln (N)}{\omega} \right)^{1/\omega}.
\]
\end{lemma}

\begin{proof}
Consider such an $N$, and write $2 \leq p_1 < p_2 < \dotsc < p_{\omega}$ upon ordering. Since $p_i \ge 2$, the first $\omega - 1$ terms give $\prod_{i =1 }^{\omega} \ln(p_i) \ge (\ln(2))^{\omega - 1}$. Additionally, the maximal prime factor satisfies $p_{\omega} \ge N^{1/\omega}$, such that $\ln(p_{\omega}) \ge \ln(N)/\omega$; combining these results gives $\prod_{i = 1}^{\omega} \ln(p_i) = \ln(p_{\omega}) \prod_{i = 1}^{\omega - 1}\ln(p_i) \ge \ln(2)^{\omega - 1}(\ln(N)/\omega)$; the result follows.
\end{proof}

We now utilize this result to establish the proof of the finer $\lim\sup$ condition on $q_C$.

\begin{proof}[Proof of Theorem \ref{T-szpiro}]
Let $(a_n, b_n, c_n) \in \mathcal{F}$, and suppose the desired conditions regarding $c_n$ and $\omega_n$ are satisfied. Substituting the bound provided by Lemma \ref{sqfree-lemma} into equation (\ref{eq_primefact}) gives
\begin{align*}
    q_C(a_n, b_n, c_n; \alpha, \beta) &\leq \frac{\ln (c_n)}{\omega_n^{\alpha\beta} \left( (\ln 2)^{1 - 1/\omega_n} \left(\frac{\ln N_n}{\omega_n}\right)^{1/\omega_n} \right)^\beta}\\ & = \frac{\omega_n^{\beta/\omega_n - \alpha\beta}}{(\ln (2))^{\beta(1 - 1/\omega_n)}} \cdot \frac{\ln (c_n)}{(\ln (N_n))^{\beta/\omega_n}}.
\end{align*}
By the bound on $\ln(c_n)$, rearrangement gives $\ln N_n \ge \frac{\ln (c_n) - \lambda \omega_n \ln (\omega_n)}{\sigma}$, such that
\begin{align*}
    q_C \le \frac{\sigma^{\beta/\omega_n}}{(\ln (2))^{\beta(1 - 1/\omega_n)} \omega_n^{\alpha\beta - \beta/\omega_n}} \cdot \frac{\ln (c_n)}{(\ln c_n - \lambda \omega_n \ln (\omega_n))^{\beta/\omega_n}}.
\end{align*}
Let us now evaluate the asymptotic limit of this bound as $n \ra \infty$ under the assumption $\omega_n \sim \delta (\ln c_n)^\gamma$ as $n \to \infty$. 

We firstly handle the constant factors. Since $\gamma > 0$, we have $\omega_n \ra \infty$ as $n \ra \infty$, such that $\lim_{n \to \infty} \frac{\beta}{\omega_n} = 0$. Hence, the constant terms satisfy $\lim_{n \to \infty} \sigma^{\beta/\omega_n} = 1$ and $\lim_{n \to \infty} (\ln 2)^{\beta(1 - 1/\omega_n)} = (\ln 2)^\beta$.

The rightmost quotient requires more careful analysis. Rewriting the compensatory scaling as
\[
    (\ln c_n - \lambda \omega_n \ln \omega_n)^{\beta/\omega_n} = \exp\left( \frac{\beta}{\omega_n} \ln(\ln c_n - \lambda \omega_n \ln \omega_n) \right),
\]
we note that since $\gamma < 1$, the value of $\omega_n \ln \omega_n \sim \delta \gamma (\ln c_n)^\gamma \ln(\ln(c_n))$ is sub-linear relative to $\ln(c_n)$. Hence, $\ln(\ln c_n - \lambda \omega_n \ln \omega_n) \sim \ln \ln c_n$. Thus, we have
\[
    \lim_{n \to \infty} \frac{\beta \ln \ln c_n}{\delta (\ln c_n)^\gamma} = 0,
\]
such that the scaling of the entire quotient, as $n \ra \infty$, satisfies $\frac{\ln (c_n)}{(\ln c_n - \lambda \omega_n \ln (\omega_n))^{\beta/\omega_n}} \ra \ln(c_n) (1 + o(1))$. Combining these results, we evaluate
\begin{align*}
    q_C(a_n, b_n, c_n; \alpha, \beta) &\lesssim \frac{1}{(\ln 2)^\beta} \cdot \frac{1}{(\delta (\ln c_n)^\gamma)^{\alpha\beta}} \cdot \ln (c_n) \\ &= \frac{1}{\delta^{\alpha\beta} (\ln 2)^\beta} (\ln (c_n))^{1 - \gamma \alpha \beta},
\end{align*}
with the bound being sharp as $n \ra \infty$ (due to the $o(1)$ term). Hence, when $\alpha \beta = 1/\gamma$, the $\ln(c_n)$ term vanishes, such that the $\lim\sup$ equals the constant term $\frac{1}{\delta^{1/\gamma} (\ln(2))^{\beta}}$. The result follows quickly for $\alpha \beta > 1/\gamma$.
\end{proof}
We remark that the $\alpha \beta < 1/\gamma$ case is less straightforward, since our derivation solely provides an upper bound. However, if there exists a subsequence of $\mathcal{F}$ such that the first $\omega_n$ primes are as small as possible, then the $i$-th prime in the expansion satisfies $\ln(p_i) \sim \ln(i \ln(i)) \sim \ln(i)$, such that $G_n := \left(\prod_{i = 1}^{\omega_n} \ln(p_i) \right)^{1/\omega_n} \sim \left( \prod_{i = 1}^{\omega} \ln(i)\right)^{1/\omega_n} \sim \ln(\omega_n)$ as a heuristic\footnote{Note that $\frac{1}{\omega_n} \sum_{i = 1}^{\omega_n} \ln(\ln(i)) \sim \frac{1}{\omega_n} \int_2^{\omega_n} \ln(\ln(x)) \, dx$. Integration and approximation of the $\int_2^{\omega_n} 1/\ln(x) \, dx$ term as $\mathrm{li}(\omega_n)\sim \omega_n/\ln(\omega_n)$ by the prime number theorem gives $\frac{1}{\omega_n} \sum_{i = 1}^{\omega_n} \ln(\ln(i)) \sim \ln \ln \omega_n$, and exponentiation gives the desired relationship.}. Thus, since $\omega_n \sim \delta \cdot (\ln(c_n))^{\gamma}$ by assumption, $G_n \sim \ln(\ln(c_n))$, which gives
\begin{align*}
    q_C \gtrsim \frac{\ln(c_n)}{\delta^{\alpha \beta} (\ln(c_n))^{\gamma \alpha \beta} (\ln(\ln(c_n)))^{\beta}} = \frac{\ln(c_n)^{1 - \gamma \alpha \beta}}{\delta^{\alpha \beta} (\ln(\ln(c_n)))^{\beta}},
\end{align*}
which diverges. Hence, if such a suitable family exists, the result may be extended to this domain as per intuition.

%The classical $abc$-quality of an $abc$-triple $(a,b,c)$ is governed by the ratio $\mathcal{R} = \frac{\ln c}{\ln \operatorname{rad}(abc)}$. By mapping our generalized quality index $q_C$ directly into this classical ratio, we uncover a hidden, exact constraint on the distribution of the bad reduction primes: the ratio of their geometric mean to their arithmetic mean. This proves that for an infinite family of Frey curves to violate the $abc$-conjecture, the prime factors must be geometrically balanced; a single massive prime outlier unconditionally forces the efficiency to zero.

We now seek to invoke this trichotomy to establish a connection between the \textit{standard} quality (Def. \ref{D_standard_quality}) and the metrics of our DGM class (Def. \ref{D_DGM_class}). To do so, we first invoke a lemma that relates the geometric and arithmetic means of distinct primes, which studies the effect of an outlier in the geometric \textit{packing} of the primes.

\begin{lemma}[Packing Primes]
\label{l_pack_primes}
Let $N = \prod_{i=1}^\omega p_i$ for primes $p_i$, corresponding to the conductor of the Frey curve $E_{a,b,c}$, with $\omega \ge 2$. Let $A := \frac{1}{\omega} \sum_{i = 1}^{\omega} \ln(p_i), G := \left( \prod_{i=1}^\omega \ln p_i \right)^{1/\omega}$. Defining the \emph{packing efficiency} as $\eta = \frac{G}{A}$, we have $0 < \eta < 1$; additionally, if the conductor satisfies $N = P \cdot N_0$, where $P$ satisfies $P \ra \infty$ while $N_0$ remains invariant, then $\eta$ decays at the same rate: $\eta = \mathcal{O}\left( (\ln P)^{\frac{1}{\omega} - 1} \right)$, and vanishes as $P \ra \infty$.
\end{lemma}

\begin{proof}
We note that the notation $A, G$ derives from the arithmetic and geometric mean of the logarithmic primes. As such, the bound $0 < \eta < 1$ follows from the AM-GM inequality (with equality not being satisfied since the primes $p_i$ are distinct).

Let us now write $N = P \cdot N_0$ by assumption. Let $G_0$ be the geometric mean of the logarithms of the $\omega - 1$ primes dividing $N_0$. Then, we may write $A = \frac{\ln(N)}{\omega} = \frac{\ln(P) + \ln(N_0)}{\omega}$, extracting the role of $P$, and, similarly, $G = (\ln P)^{1/\omega} \cdot G_0^{1 - 1/\omega}$. Substituting these into the definition of $\eta$ yields:
\[
    \eta = \frac{\omega \cdot G_0^{1 - 1/\omega} \cdot (\ln P)^{1/\omega}}{\ln (P) + \ln (N_0)} =  \left( \omega \cdot G_0^{1 - 1/\omega} \right) \cdot \frac{(\ln (P))^{1/\omega - 1}}{1 + \frac{\ln (N_0)}{\ln P}}
\]
Because $N_0$ and $\omega$ are bounded, the prefactor is constant, and the denominator limits to $1$ as $P \ra \infty$. Hence, $\eta = \mathcal{O}\left( (\ln( P))^{\frac{1}{\omega} - 1} \right)$. Since $\omega \ge 2$, $\eta \ra 0$ as $P \ra \infty$.
\end{proof}

Motivated by this \textit{packing efficiency} as a comparative between the two metrics, we establish the following.

\begin{theorem}[Factorization of $q_s$]
\label{t-big}
Let $\mathcal{F} = \{(a_n, b_n, c_n)\}_{n=1}^\infty$ be a sequence of $abc$-triples, and let $N_n := \operatorname{rad}(a_n b_n c_n)$. For any $\alpha > 0$, the standard quality $q_{s,n} := \frac{\ln (c_n)}{\ln (N_n)}$ satisfies
\[
    q_{s,n} = q_C(a_n, b_n, c_n; \alpha, 1) \cdot \omega_n^{\alpha - 1} \cdot \eta_n
\]
Consequently, if the sequence occupies the boundary domain, such that $\omega_n \sim \delta (\ln c_n)^\gamma$ with $\alpha = 1/\gamma$, we have
\[
    \limsup_{n \to \infty} q_{s,n} \le \left( \frac{1}{\delta^{1/\gamma} \ln(2)} \right) \cdot \limsup_{n \to \infty} \left( \omega_n^{\frac{1-\gamma}{\gamma}} \cdot \eta_n \right)
\]
\end{theorem}

\begin{proof}
Since the DGM quality class satisfies $q_C(a_n, b_n, c_n; \alpha, 1) = \frac{\ln c_n}{\omega_n^\alpha G_n}$, per our definition of $G_n$, rearrangement gives $\ln c_n = q_C(a_n, b_n, c_n; \alpha, 1) \cdot \omega_n^\alpha \cdot G_n$. Therefore, by the definition of $A_n$ as in Lemma \ref{l_pack_primes}, we have $\ln N_n = \omega_n A_n$. Substitution into the standard quality in Def. \ref{D_standard_quality} gives
\[
    q_{s,n} = \frac{\ln (c_n)}{\ln (N_n)} = \frac{q_C(a_n, b_n, c_n; \alpha, 1) \cdot \omega_n^\alpha \cdot G_n}{\omega_n \cdot A_n}
\]
Hence, substitution of the packing efficiency $\eta_n$ immediately gives
\begin{align*}
    q_{s,n} = q_C(\cdot; \alpha, 1) \cdot \omega_n^{\alpha - 1} \cdot \eta_n.
\end{align*}
In the critical boundary, we assume that the prime count scales as $\omega_n = \delta (\ln c_n)^\gamma \bigl(1 + o(1)\bigr)$, such that $\alpha = 1/\gamma$. By Theorem \ref{T-szpiro}, the limit supremum of $q_C(a_n, b_n, c_n; 1/\gamma, 1)$ converges exactly to the boundary constant $\frac{1}{\delta^{1/\gamma} \ln 2}$. Hence, we have
\[
    \limsup_{n \to \infty} q_{s,n} \le \left( \frac{1}{\delta^{1/\gamma} \ln 2} \right) \cdot \limsup_{n \to \infty} \left( \omega_n^{1/\gamma - 1} \cdot \eta_n \right),
\]
as desired.

As an extension, if, additionally, the largest prime factor 
$P_n := P(N_n)$ satisfies $\ln P_n \in O(\ln c_n)$ (i.e., $P_n \sim c_n^\kappa$ 
for some $\kappa > 0$), then by Lemma~\ref{l_pack_primes}, we have $\eta_n = \mathcal{O}\!\left((\ln c_n)^{1/\omega_n - 1}\right),$ and hence $ \omega_n^{\frac{1-\gamma}{\gamma}}\cdot\eta_n 
=\mathcal{O}\left((\ln c_n)^{1/\omega_n - \gamma}\right) \ra 0$ as $n\to\infty$, since $\gamma > 0$. Consequently, we must have $ \limsup_{n\to\infty}\, q_{s,n}= 0.$
\end{proof}

Hence, our phase transitions build towards a correspondence between the \textit{standard} quality and the DGM quality class, mediated by the primes appearing in the factorization of families of triples.

\subsection{Sharp Packing Bounds}
  \label{S_packing_abc}

  The factorization of Theorem \ref{t-big} reduces at $\alpha=1$ to the exact identity
  $q_s = \eta\,q_{\mathrm{DGM}}$, with $\eta=G/A$ the packing efficiency of Lemma \ref{l_pack_primes}. Since $q_{\mathrm{DGM}}$ diverges (Theorem \ref{C_abc_DGM}) while $q_s$
  is bounded \textit{conjecturally}, the resulting comparison is governed by the decay of $\eta$; thus, extracting analytic content from this bound requires a lower bound on $\eta$ – which, in principle, serves as a reverse AM-GM inequality. We supply such a bound via the Specht ratio, which gives a restatement of the $abc$-conjecture in terms of the DGM quality.

  \begin{lemma}[Specht Implications]\label{T_specht}
  Let $(a,b,c)$ be an $abc$-triple with $\omega\ge 2$ and maximal prime divisor $P$,
  and set $h:=\ln(P)/\ln (2)$ and $S(h) = \frac{(h-1)\,h^{1/(h-1)}}{e\ln h}$ as the Specht ratio, for $h > 1$; we define $S(1) = 1$. Then, $q_s(\cdot) = \eta\,q_{\mathrm{DGM}}(\cdot)$ with $1/S(h)\le \eta\le 1$, and therefore
  \[
      q_s(a,b,c)\;\le\; q_{\mathrm{DGM}}(a,b,c)\;\le\; S(h)\,q_s(a,b,c).
  \]
  Observe $S(h)=\frac{h}{e\ln h} \cdot (1+o(1))$ for large $h$.
  \end{lemma}

  \begin{proof}
  The identity $q_s=\eta\,q_{\mathrm{DGM}}$ is the case $\alpha=1$ of Theorem \ref{t-big}, since $q_C(\cdot;1,1)=q_{\mathrm{DGM}}(\cdot)$. The upper
  bound $\eta\le 1$ is given by the AM-GM inequality, as noted in Lemma
  \ref{l_pack_primes}. For the lower bound, we apply Specht's inequality to the
  positive reals $\ell_i := \ln p_i$, which lie in $(\ln 2,\ln P)$ with maximal comparative ratio being bounded by $\ln(P)/\ln(2) = h$ by monotonicity \cite{specht1960theorie}. The arithmetic ($A$) and geometric ($G$) means of the $\ell_i$ satisfy $A/G \leq S(h)$, such that $\eta = G/A \ge 1/S(h)$. Substituting $\eta\in(1/S(h),1)$ into $q_s=\eta\,q_{\mathrm{DGM}}$ gives $q_s\le q_{\mathrm{DGM}}\le S(h)\,q_s$.
  \end{proof}

  \begin{theorem}[Packing-efficiency formulation of $abc$]\label{T_abc_eta}
  The $abc$-conjecture (Conjecture \ref{C_abc_conjecture}) is equivalent to the assertion that, for every $\varepsilon>0$,
  \[
      q_{\mathrm{DGM}}(a,b,c)\;\le\;\frac{1+\varepsilon}{\eta(a,b,c)}
  \]
  for all but finitely many $abc$-triples.
  
  Thus, the $abc$-conjecture implies that for every $\varepsilon>0$ and all but finitely many triples,
  \[
      q_{\mathrm{DGM}}(a,b,c)\;\le\;(1+\varepsilon)\,S\!\Bigl(\tfrac{\ln (P)}{\ln (2)}\Bigr)
      \;=\;(1+\varepsilon)\,\frac{\ln (P)}{e\ln (2) \cdot \,\ln\ln P}\,(1+o(1)),
  \]
  so that $q_{\mathrm{DGM}} \in \mathcal{O}\!\bigl(\ln P(abc)/\ln\ln P(abc)\bigr)$.
  \end{theorem}

  \begin{proof}
  By Lemma \ref{T_specht}, $q_s=\eta\,q_{\mathrm{DGM}}$ exactly; thus, the inequalities $q_s\le 1+\varepsilon$ and $q_{\mathrm{DGM}}\le(1+\varepsilon)/\eta$ are equivalent over each $abc$-triple. The former inequality holds for all but finitely many triples precisely when $\limsup_{c\to\infty}q_s(a, b, c) \le 1$, which establishes the equivalence. Now, let us assume $abc$. Then for all but finitely many triples, $q_s(a,b,c) \le 1+\varepsilon$, and using $\eta\ge 1/S(h)$ from Lemma \ref{T_specht},
  \[
      q_{\mathrm{DGM}} (a, b, c)=\frac{q_s (a, b, c)}{\eta (a, b, c)}\le\frac{1+\varepsilon}{\eta (a, b, c)}\le(1+\varepsilon)\,S(h),
  \]
  with the asymptotic result following from the bound on $S(h)$ for large $h$. Since $P\le c$ forces $\ln P \leq \ln c$, the bound $q_{\mathrm{DGM}} \in \mathcal{O}(\ln P/\ln\ln P)$ follows.
  \end{proof}

\subsection{Power-Mean Interpolation}\label{S_qt_interpolation}

As a last piece of analysis on the DGM quality itself, we note that the standard quality $q_s$ and the DGM quality $q_{\mathrm{DGM}}$ are the endpoints of a one-parameter family that is obtained by replacing the geometric mean in the denominator of $q_{\mathrm{DGM}}$ with a generalized power mean (effectively, a $p$-norm). Indeed, for $t\in[0,1]$, let us set $\kappa =1-t$ and let $M_{\kappa}$ be the power mean of the set $\{\ln p_i\}_{i=1}^\omega$, such that
\[
M_{\kappa}(\{\ln p_i\}) = \begin{cases}\bigl(\tfrac{1}{\omega}\sum_i(\ln p_i)^{\kappa} \bigr)^{1/\kappa} & 0<\kappa \le 1\\[2pt] \sqrt[\omega]{\prod_i\ln p_i} & \kappa =0,\end{cases}
\]
such that
\[
q_t(a,b,c) \;=\; \frac{\ln c}{\omega\cdot M_{1-t}(\{\ln p_i\})}.
\]
satisfies that $q_0 = q_s$ and $q_1 = q_{\mathrm{DGM}}$ (since $M_1$ recovers the arithmetic mean, and $M_0$ recovers the geometric mean). By Maclaurin's inequality, $M_{\kappa}$ is non-decreasing in ${\kappa}$, hence $q_t$ is non-decreasing in $t$ on every triple; this establishes $q_s\le q_{\mathrm{DGM}}$ pointwise.

We note, briefly, that for our $(\omega,k)$-families, as defined in Section \ref{S_pt}, specific bounds may be generated on the limiting value of $q_t$ on these families. Indeed, for $t \in [0, 1)$, the same methods as in Section \ref{S_pt} give that $q_t(a_n, b_n, c_n) \ra 2^{1/(1 - t)}/2$ on Mersenne triples $a_n, b_n, c_n$ (which have $(\omega, k) = (2, 1)$), and $q_t(a_n, b_n, c_n) \ra (3/2)^{1/(1 - t)}/3$ on twin prime triples (which have $(\omega, k) = (3, 2)$). In general, we have $q_t(a_n, b_n, c_n) \ra (\omega/k)^{1/(1 - t)}/\omega$ for all such $\omega, k, t$, and corresponding families.\footnote{Explicitly, this occurs since we have $k$ primes growing as $\Theta(\ln(c))$ while the other $\omega - k$ primes are bounded, such that $(M_{\kappa}(\ln(\{p_i\}))^{\kappa} = 1/\omega (k (\ln(c))^{\kappa} + O(1))$, and $M_{\kappa} (\ln(\{p_i\})) \sim (k/\omega)^{1/\kappa} \ln(c)$. Substitution into $q_t(\cdot)$ yields the limiting form.} Hence, divergence only occurs as $t \ra 1$ from below; this occurs since, for $\kappa > 0$, $M_{\kappa}$ is dominated by the largest term $(\ln(c))^{\kappa}$, except at $\kappa = 0$.

One might consider the question of when $\limsup_{c \ra \infty} q_t(a, b, c)$ is attainable for specific $t \in (0,1)$ and choices of triples. To attain this, we may consider triples $(1, p - 1, p)$, such that the number of primes $\omega$ diverges while the primes dividing $p-1$ themselves grow slowly. For instance, for triples on which $\omega(p-1)\sim\ln\ln p$ and $P(p-1)$ is bounded by $p^{\delta(p)}$ for some function $\delta$, the formulation of $q_t(\cdot)$ gives
\begin{equation*}
q_t(1,p-1,p) \;=\; \frac{\omega^{1/(1-t)-1}}{\bigl((\omega-1)\delta(p)^{1-t}+1\bigr)^{1/(1-t)} (1+o(1))};
\end{equation*}
thus, if $\delta$ grows as $\mathcal{O}(\omega^{-1/(1-t)})$, we have $q_t\sim(\ln\ln p)^{t/(1-t)}$, which diverges as $p \ra \infty$. The former condition requires $\delta(p) \ra 0$ at the rate 
$\delta(p)=(\ln\ln p)^{-1/(1-t)}(1+o(1))$, which appears to be beyond the strongest known unconditional results, which yield $P(p-1) \leq p^{0.5366}$ with $\omega$ bounded \cite{baker2014}.

%Conversely, the $abc$-conjecture itself gives \textit{upper} bounds on $q_t$ that depend on $\omega$ and $t$. By Maclaurin's inequality $M_{1-t}\le M_1=\mathrm{AM}$, and concentrating mass on a single large prime maximizes $\mathrm{AM}/M_{1-t}$; explicit computation yields
%\frac{M_1}{M_{1-t}} \;\le\; \omega^{t/(1-t)}
%for any $\omega$-tuple of positive reals $\ge\ln 2$. Combined with $q_s\le 1+\epsilon$ under abc:
On the other hand, the $abc$-conjecture itself gives \textit{upper} bounds on $q_t$, for fixed $\omega$.
\begin{theorem}\label{T_qt_UB}
Assuming the $abc$-conjecture, for every $\eps>0$, $t\in[0,1)$, $\omega_0\ge 2$, and $abc$-triple with $\omega(abc)=\omega_0$, we have
\[
\limsup_{c\to\infty,\,\omega=\omega_0} q_t(a,b,c) \leq (1+\eps) \cdot \omega_0^{t/(1-t)}.
\]
\end{theorem}
\begin{proof}
Fix $\eps > 0$, $t \in (0,1)$. Then, we know $q_t(a, b, c) = q_s(a, b, c) \cdot (M_1/(M_{1 - t}))$ by definition. Additionally, we have $M_1/M_{1-t} \le \omega_0^{t/(1-t)}$, 
since applying the $(\ell^1-\ell^{1 - t})$-norm inequality to the vector 
$(\ln p_1, \ldots, \ln p_{\omega_0})$ gives
\[
    \sum_{i=1}^{\omega_0} \ln p_i \;\le\; \omega_0^{t/(1-t)} 
    \left(\sum_{i=1}^{\omega_0}(\ln p_i)^{1-t}\right)^{1/(1-t)};
\]
dividing both sides by $\omega_0$ yields $M_1 \le \omega_0^{t/(1-t)} M_{1-t}$. 
Therefore, combined with the $abc$-conjecture giving $q_s \le 1 + \eps$, the bound 
$q_t \le (1+\eps)\cdot\omega_0^{t/(1-t)}$ follows. Note that, in particular, for bounded $\omega$, $q_t(\cdot)$ is bounded by an explicit function of $t$, increasing pointwise as $t \ra 1$.
\end{proof}

\subsection{Experimental Evaluation}
\label{Data_class}

We present selected experimental results corresponding to the DGM quality class. Figure \ref{Class_1} compares $q_C(\cdot;\alpha,1)$ for $\alpha\in\{0,0.5,1\}$. Starred points are triples $(a, b, c_0; \cdot)$ with higher quality than any others with the same $c_0$ entry. We note $\alpha=0.5$ privileges triples of moderate roundness. Table \ref{Class_merge} demonstrates \textit{record-setting triples} $(1, b, c_0)$; that is, triples with strictly larger quality than any triples with $c < c_0$. We note that decreasing $\alpha$ empirically gives strictly smaller record-setting triples. Finally, Table \ref{Class_8_short} demonstrates that increasing $\beta$ produces nested subsets of record-setting triples, when restricted to the form $(1, b, c)$.

\begin{figure}[t]
\centering
\includegraphics[width=0.7\columnwidth]{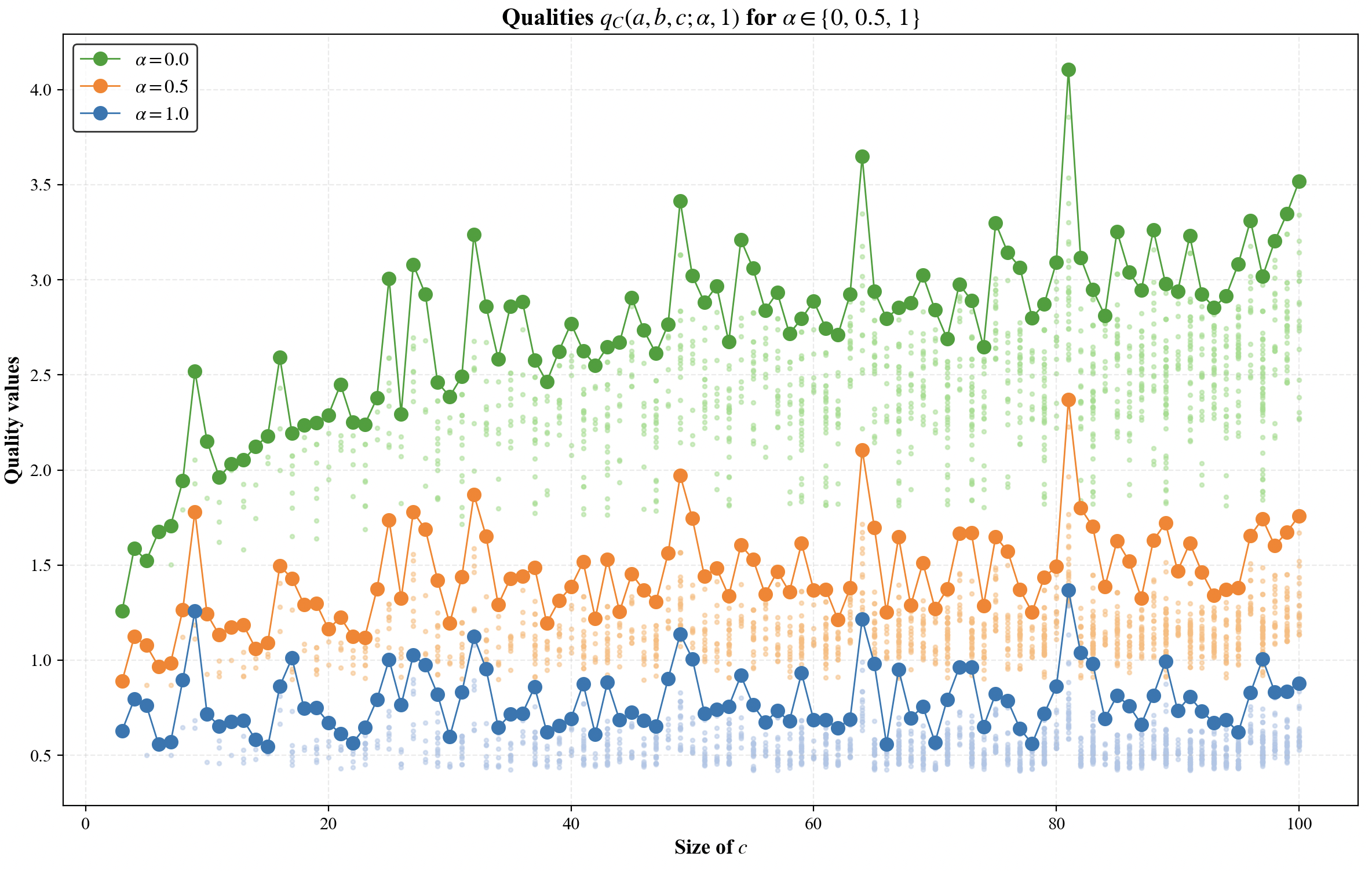}
\caption{Illustration of $q_C(a, b, c;\alpha,1)$ for $\alpha=0,0.5,1$, with highest-qualities for a given $c$ interpolated. \label{Class_1}}
\end{figure}

\begin{table}[t]
\centering\small
\begin{tabular}{|r|r|c|c|c|}
\hline
$b$ & $c$ & $q_C(\cdot; 0,1)$ & $q_C(\cdot; 0.5,1)$ & $q_C(\cdot; 1,1)$ \\
\hline
8     & 9     & 2.518 & 1.792 & 1.259 \\
80    & 81    & 4.106 & 2.678 & 1.369 \\
242   & 243   & 4.494 & 3.017 & 1.498 \\
512   & 513   & 4.768 & 3.239 & 1.589 \\
4374  & 4375  & 6.746 & 4.253 & 1.687 \\
8191  & 8192  & ---   & ---   & 1.803 \\
62207 & 62208 & ---   & ---   & 1.810 \\
65536 & 65537 & ---   & ---   & 2.000 \\
131071 & 131072 & --- & ---   & 2.062 \\
\hline
\end{tabular}
\caption{Record-setting triples $(1,b,c)$for $c\le 200000$ and $\beta=1$; triples that are no-longer record-setters are ---.}
\label{Class_merge}
\end{table}

\begin{table*}[t]
\centering\small
\begin{tabular}{|r|r|c|c|c|c|c|c|c|}
\hline
$b$ & $c$ & $\beta=0.75$ & $\beta=1$ & $\beta=1.25$ & $\beta=1.5$ & $\beta=1.75$ & $\beta=2$ & $\beta=3$ \\
\hline
2     & 3     & 1.217 & 1.259 & 1.303 & 1.348 & 1.394 & 1.443 & 1.653 \\
8     & 9     & 2.434 & 2.518 & 2.605 & 2.695 & 2.789 & 2.885 & 3.307 \\
80    & 81    & 4.177 & 4.106 & 4.037 & 3.970 & 3.903 & 3.837 & 3.586 \\
675   & 676   & 5.257 & 4.894 & 4.556 & 4.241 & 3.948 & ---   & ---   \\
2400  & 2401  & 6.613 & 6.264 & 5.932 & 5.619 & 5.322 & 5.040 & 4.056 \\
4374  & 4375  & 7.123 & 6.746 & 6.390 & 6.052 & 5.732 & 5.429 & 4.369 \\
71874 & 71875 & 8.015 & 7.172 & 6.418 & ---   & ---   & ---   & ---   \\
\hline
\end{tabular}
\caption{Record-setting triples $(1,b,c)$, $c\le 200{,}000$, for varying $\beta$ at $\alpha=1$. Larger $\beta$ yields strictly smaller record-setting sets of triples, as triples that are no longer record-setters are given ---.}
\label{Class_8_short}
\end{table*}
\iffalse
\section{Incorporating Smoothness}
\label{S_Smoothness}

\subsection{Motivation and Definition}

Some DGM-Class metrics admit large primes in their record triples (such triples are not \emph{smooth}). We divide by the largest prime divisor, raised to a power $\psi$, to penalize large primes.

\begin{definition}[Smooth DGM Quality Class]\label{D_Smooth_Class}
Let $\alpha\ge 0$, $\beta>0$, $\psi\ge 0$. With $P,\omega,U$ as in Def.\ \ref{D_DGM_class} and $p_{\max}=\max(P)$,
\[
q_{SC}(a,b,c;\alpha,\beta,\psi) = \frac{q_C(a,b,c;\alpha,\beta)}{p_{\max}^\psi}.
\]
\end{definition}

For $\psi=0$, $q_{SC}=q_C$. For $\psi\ge 0.5$, smoothness dominates; for $\psi>1$, quality values shrink rapidly. The metric $q_{SC}(\cdot;0.5,1,0.25)$ exhibits asymptotic behavior qualitatively similar to the standard quality $q_s$ (see Figure \ref{Smooth_2}).
\fi

\section{Triple-Finding Algorithms}
\label{S_Algorithms}

Within this section, we briefly present approaches to the calculation of high-quality triples for the DGM quality; we hope that our methods, particularly in Section \ref{S_cyclotomic}, might generalize to other manners of determining high-quality triples computationally for variant quality measures. We set the maximum value of $c$ for triples $(a, b, c)$ to be $h$, such that its digits are bounded as $d = \lfloor \log_{10}(h)\rfloor + 1$.

\subsection{Brute Force Algorithm}

The naïve approach enumerates all $abc$-triples with $c\le h$ and $a\le\lfloor c/2\rfloor$; computing the resulting prime factorizations allows for the determination of $q_{\text{DGM}}(a, b, c)$.

\begin{algorithm}[!t]
\caption{Brute Force Method}\label{A_Naive}
\begin{algorithmic}[1]
\For{$c=3$ to $h$}
  \For{$a=1$ to $\lfloor c/2\rfloor$ with $\gcd(a,c)=1$}
    \State $b\gets c-a$
    \State $\text{primes}\gets \bigcup_{i \in \{a, b, c\}} \textsc{primefact}(i)$
    \State Compute $q_{\mathrm{DGM}}(a,b,c)$ via \eqref{eq3}
    \State Determine if $q_{\mathrm{DGM}}(a, b, c)\ge q_{\min}$
  \EndFor
\EndFor
\end{algorithmic}
\end{algorithm}
For each $c_0\ge 3$, the number of triples $(a,b,c_0)$ with $a\le\lfloor c_0/2\rfloor$ is $\varphi(c_0)/2$. Summing via the totient summatory function $\Phi(n)=\sum_{k=1}^n\varphi(k)\sim n^2/(2\zeta(2))=3n^2/\pi^2$, the total triple count is $(\Phi(h)-2)/2>0.15h^2$, asymptotically \cite{hardywright}. Each triple requires three prime factorizations; hence, even with sieve methods for factorization, the algorithm is $\Omega(10^d)$. Our implementation struggled to find triples with $q_{\mathrm{DGM}} > 1.6489$, with $h > \mathcal{O}(10^6)$ being infeasible. Hence, \textit{families} of triples provide significantly more tractable methodologies, as below.

\subsection{Power of $p,q$ Algorithm}
Motivated by the power of $2, 3$ family in Section \ref{S_DGM_Quality}, we introduce the following algorithm.

\begin{algorithm}[!t]
\caption{Power of $p,q$ Method}\label{A_p_q}
\begin{algorithmic}[1]
\State $i_{\max} \gets \lceil\log_p(h/2)\rceil$, $j_{\max} \gets \lceil\log_q(h/2)\rceil$
\For{$i=1$ to $i_{\max}$, $j=1$ to $j_{\max}$}
    \State $a \gets p^i$, $b \gets q^j$, $c \gets a+b$
    \If{$\textsc{IsPrime}(c)$}
        \State Compute $q_{\mathrm{DGM}}(a,b,c)$ using \eqref{eq3}
        \State Determine if $q_{\mathrm{DGM}} \ge q_{\min}$
    \EndIf
\EndFor
\end{algorithmic}
\end{algorithm}
The total number of triples considered in this manner is $i_{\max}\cdot j_{\max}=O(\log^2 h)=O(d^2)$, which is polynomial in $d$; however, each triple requires a factorization of $c$ to compute the quality, which remains exponential in $d$. The algorithm remains costly, since testing with the AKS primality test remains $O(d^6)$; we observed that $h=10^9$ remains tractable, such that triples with $q_{\mathrm{DGM}}(a, b, c)>10$ may be determined \cite{agrawal2004primes}.

\subsection{Mersenne Prime Algorithm}

Motivated by the consideration of families of triples, Algorithm \ref{A_2_3} uses the 51 known Mersenne primes; by Theorem \ref{T_3_4}, no factorizations are required \cite{mersennelist}.

\begin{algorithm}[!t]
\caption{Mersenne Prime Algorithm}\label{A_2_3}
\begin{algorithmic}[1]
\State $n_{\max}\gets\log_2(h)$
\State $\text{MersenneList}\gets\text{precomputed}$
\For{each $n\in\text{MersenneList}$ with $n\le n_{\max}$}
  \State $b\gets 2^n-1$;\quad $c\gets 2^n$
  \State $q_{\mathrm{DGM}}\gets \sqrt{n}/2$ \Comment{Theorem \ref{T_3_4}}
  \If{$q_{\mathrm{DGM}}\ge q_{\min}$}
    \Return $(1,b,c,q_{\mathrm{DGM}})$
  \EndIf
\EndFor
\end{algorithmic}
\end{algorithm}

This algorithm remains $O(d)$, and finds Mersenne triples with up to 24 million digits and qualities of up to $\approx 4543.95$, as motivated by the high-quality considerations. Our methodology, however, can be generalized using the cyclotomic methodology that follows.

\subsection{Cyclotomic Algorithm}\label{S_cyclotomic}

A theoretical limitation of Algorithm \ref{A_2_3} is that it restricts to a single, highly restricted family. However, Theorem \ref{T_phase_transition} predicts that there exists an infinite family exhibiting distinct asymptotics with respect to critical exponent $\beta_0$, each for a corresponding $(\omega,k)$; thus, we may replace the base $b = 2$ in the Mersenne prime construction with another base $b \in \{2, 3, 5, 6, 7, 10, 11, \dotsc\} = \Z - \mathcal{P}$, where $\mathcal{P}$ is the subset of prime powers in $\Z$. We now present a new algorithm that exploits this, motivated by the cyclotomic structure of $\Z_{b^n-1}$.

The theoretical basis of our algorithm is as follows. For any base $b \ge 2$, $b^n - 1 = \prod_{d\mid n} \Phi_d(b)$, where $\Phi_d$ is the $d$-th cyclotomic polynomial. By Zsigmondy's theorem, for every $n>1$ (with the exception of $b=2,n=6$) the factor $\Phi_n(b)$ has a primitive prime divisor (that is, $p$ such that $p$ does not divide $b^m-1$ for any $m<n$). Consequently, $\omega(b^n - 1) \ge |\{d: d \mid n, \Phi_d(b) > 1\} \gtrsim \tau(n)$, for the divisor function $\tau(n)$. Thus, triples of the form $(1,b^n-1,b^n)$ that minimize $\omega$ arise when $n$ has few divisors; thus, the minimum is attained when $\Phi_n(b)$ is itself prime and $n\in\mathbb{P}$ (so that $b^n-1=(b-1)\Phi_n(b)$). For $b > 2$, this yields the set of \emph{generalized repunit primes} $(b^n-1)/(b-1)$ that are given in the Cunningham project \cite{mersennelist}. Since the Cunningham project provides precomputed factorizations of $b^n\pm 1$ for thousands of exponents $n$ and $b \neq 12$, which give the values of $\omega(b^n-1)$ in $O(1)$ time, the per-triple cost of evaluating $q_{\mathrm{DGM}}$ is $O(1)$, identically to Algorithm \ref{A_2_3}. Thus, the asymptotic complexity of this algorithm remains $O(d)$, although $b$ may itself vary   33.

\begin{algorithm}[!t]
\caption{Cyclotomic Algorithm\label{A_cyclo}}
\hspace*{\algorithmicindent}\textbf{Input:} $h$, base list $B\subseteq\{2,3,5,6,7,\dots\} \subset \Z - \mathcal{P}$.
\begin{algorithmic}[1]
\For{$b\in B$}
  \State $n_{\max}\gets\lfloor\log_b h\rfloor$
  \State $T_b\gets\textsc{CunninghamTable}(b)$ \Comment{known $\omega(b^n-1)$}
  \For{$n$ with $1\le n\le n_{\max}$ and $n\in T_b$}
    \State $a\gets 1$;\quad $b'\gets b^n-1$;\quad $c\gets b^n$
    \State $\{\omega,\text{primes}\}\gets T_b[n]$ \Comment{$O(1)$ lookup}
    \State Compute $q_{\mathrm{DGM}}$ via \eqref{eq3}
    \If{$q_{\mathrm{DGM}}\ge q_{\min}$}
      \State return $(1, b^n - 1,b^n , q_{\mathrm{DGM}})$
    \EndIf
  \EndFor
\EndFor
\end{algorithmic}
\end{algorithm}

\begin{table*}[t]
\centering
\begin{tabular}{|l|c|c|c|c|}
\hline
\textbf{Property} & \textbf{Alg.\ \ref{A_Naive} (Brute)} & \textbf{Alg.\ \ref{A_p_q} ($p,q$)} & \textbf{Alg.\ \ref{A_2_3} (Mers.)} & \textbf{Alg. \ref{A_cyclo} \textbf{(Cyclo.)}} \\
\hline
Highest $q_{\mathrm{DGM}}$ & 1.6849 & 10.0424 & 4543.9502 & $14060$ \\
Digits of $c$ in largest triple & 5 & 9 & 24{,}000{,}000 & 1{,}068{,}673 \\
Largest computed $c$ & $10{,}000$ & $924{,}291{,}401$ & $2^{82{,}589{,}933}$ & $7^{1{,}264{,}699}$ \\
Median time of execution & 16.41\,s & 0.77\,s & 0.00168\,s & $0.005\,$s \\
Fact. ($F)$/primality test. $(P)$ & $C\cdot 10^{2d} (F)$ & $C\cdot d^2 (P)$ & $0$ & $0$\\
Complexity & $O(10^d)$ & $O(d^{\tau}), \tau \leq 8$ & $O(d)$ & $O(|B|d)$ \\
\hline
\end{tabular}
\caption{Algorithm performance comparison (executed on Python 3.8.2 with a 1.1\,GHz CPU).\label{Alg_tab}}
\end{table*}

We note that, for $b = 2$, Algorithm \ref{A_cyclo} reproduces Algorithm \ref{A_2_3};
however, for $b = 3$, the analogous family corresponds to base-3 generalized repunit
primes (OEIS A028491). By Theorem \ref{T_phase_transition} with $(\omega,k)=(3,1)$,
these triples are critical at $\beta_0=3$ and give high-quality triples in the
regime where Mersenne triples have already saturated, due to a larger value of
$\omega/k$. More generally, every base $b\in\{3,5,6,7,10,11,\dots\}$ in the Cunningham
project (OEIS A028491, A004061, A004062, A004063, A004023, A005808) yields a
single-large-prime family with $\omega\in\{3,4\}$ and $k=1$, for which $q_{\mathrm{DGM}}$
grows asymptotically as $ q_{\mathrm{DGM}}(a, b, c) \sim \mathcal{O}((\ln c)^{(\omega-1)/\omega})$, which is strictly greater than the $\sqrt{\ln c}/2$ rate for Mersenne triples (which correspond to $\omega = 2$).
Thus, with respect to optimizing for the DGM Quality \textit{Class}, Algorithm
\ref{A_cyclo} populates a strictly larger portion of the $(\alpha,\beta)$-plane with
high-quality DGM triples, at the same $O(d)$ asymptotic cost. Evaluating
$q_{\mathrm{DGM}}$ on the largest currently known generalized-repunit prime in OEIS
A004063 – that is, base $b=7$ at $n=1{,}264{,}699$, with $\omega=4$ and $(\omega-1)/\omega = 3/4$ –
yields $q_{\mathrm{DGM}}\approx 14{,}060$, exceeding the largest Mersenne quality
$\sqrt{82{,}589{,}933}/2 \approx 4{,}544$, despite using a triple $(a, b, c)$ with $c$ having roughly $65$ times fewer digits.

To corroborate the $(\omega, k)$-phase transition analogy, we also may evaluate the behavior of the quality metric on such families for $\alpha = 1, \beta = 2$. Given Theorem \ref{T_phase_transition}, the Mersenne family saturates precisely at this choice, since $\omega/k = 2$ for all triples in this family; hence $q_C(a, b, c; 1, 2)$ limits to $1/(4 \ln(2)) \approx 0.36$ for triples in the Mersenne family. However, using base $6$ in Algorithm \ref{A_cyclo} (with $\omega = 3,$ such that $\beta_0 = 3$ as in Theorem \ref{T_phase_transition}) gives triples with $q_C(a, b, c; 1, 2) \approx 62.4$, base $7$ gives triples of quality $\approx 80.4$; the empirical growth of the latter family is $q_C(a, b, c; 1, 2) \sim \mathcal{O}(n^{1/2})$. 

In Table \ref{Alg_tab}, we summarize the performance of all four algorithms. The $O(d)$ complexity of the Mersenne and cyclotomic methods, together with the pure number of Mersenne triples and known generalized repunit primes, generate an advantage over Algorithms \ref{A_Naive},\ref{A_p_q}, with $\operatorname*{arg\,max}_{(a, b, c)} q_{\text{DGM}} (a, b, c)$ over our triples of consideration resulting from the base-$7$ triple discussed previously. 

\subsection{Parametric Hulls over the DGM Class}

As extensional material, we observe that the DGM class from Section \ref{S_DGM_Class} provides a continuum of metrics indexed by $(\alpha,\beta)$, such that the highest-quality triples depend on \textit{where} in this continuum of metrics the current choice of $q_{C}(\cdot; \alpha, \beta)$ lies. 
For instance, a triple ($a_0, b_0, c_0)$ of particularly high-quality – and potentially higher-quality than any triples $(a, b, c)$ with $c < c_0$ - at $\beta=1$ need not be so at $\beta=2$, as the nested record sets of Tables \ref{Class_merge}-\ref{Class_8_short} demonstrate. 
Hence, rather than rescanning the pool of triples at each parameter value, we determine high-quality triples using a structural feature of the DGM class. We write $G=(\prod_i\ln p_i)^{1/\omega}$ as in Section \ref{S_DGM_Class} and invoke equation \eqref{eq_primefact}, such that

  \[
      \ln (q_C(a,b,c;\alpha,\beta))=\ln(\ln (c)) -\beta\ln (G) -(\alpha\beta)\ln\omega,
  \]
  which, in the coordinates $(u,v):=(\beta, \alpha\beta)$, is a bilinear map in $(\alpha, \beta)$ with intercept $\ln(\ln (c))$ and gradient $(-\ln G,-\ln\omega)$. Thus, for a fixed triple $(a, b, c)$ (which fix $\omega, G$), variation of the parameters $\alpha, \beta$ over $\alpha \ge 0, \beta > 0$ gives a plane determined the corresponding parameters $\{u>0,\,v\ge 0\}$ (the image of the set $\{\beta>0,\,\alpha\ge 0\}$ under the bijection $\alpha \gets v/u, \beta \gets u$). If such a plane is embedded in a space homeomorphic to $\R^3$ with the \textit{height} of a point on the plane corresponding to the value of $ \ln (q_C(a, b, c; \alpha, \beta))$ for fixed $(a, b, c)$ and varying $\alpha, \beta$, then comparing two triples at a specific $(\alpha, \beta)$ amounts to comparing the vertical heights of the resulting planes.

  Determining \textit{optimal} choices of triples for each set of parameters $(\alpha, \beta)$ follows from this geometric construction. Since maximizing $q_C$ is functionally equivalent to maximizing
  $\ln (q_C)$, the highest attainable log-quality over a set $\mathscr{T} := \{T_i\}$ of $n$ $abc$-triples $(a_i, b_i, c_i)$ is the pointwise maximum
  \[
      F(u,v) := \max_i\bigl(\ln(\ln (c_i)) -u\ln G_i-v\ln\omega_i\bigr),
  \]
  which is precisely the upper envelope of the $n$ planes. The triple that realizes such a maximum at the point $(u,v)$ is therefore a \textit{record-setter} for the DGM quality (i.e., its quality is higher than that of any triple with smaller $c$) for such an $(\alpha,\beta)$ pair. A triple whose plane lies below the envelope everywhere is never optimal, since if the triple $T_i$ is maximized at no choice of parameter, then for every
  $(\alpha,\beta)$ some $T_j \in \mathscr{T}$ satisfies $\ln (q_C(T_j; \alpha, \beta))>\ln (q_C(T_i; \alpha, \beta))$, so $T_i$ is dominated
  within the class. The record-setting triples are therefore exactly
  those contributing a face to the broad envelope; equivalently, such triples form the vertices of the lower convex hull of the points $(\ln G_i,\ln(\omega_i),\ln\ln c_i)$. This method generates a complete set of record-setters that  supersedes the case-by-case analysis exhibited in scenarios such as Tables \ref{Class_merge}–\ref{Class_8_short}.

  Moreover we observe that projecting such a record-setting envelope onto the $(u,v)$-plane partitions the parameter region into convex domains, one per record-setting triple, across which the \textit{optimal} triple to maximize the resulting $q_C(\cdot, u, v)$ remains constant. We note that the boundaries of each domain are the precise loci at which such a record changes; these boundaries are the counterpart of the transitions in Tables \ref{Class_merge}–\ref{Class_8_short}. For triples $i,j$, the boundary between the corresponding region in the $(u, v)$ domain is given by $T_i$ and $T_j$ is
  the implicit form
  \[
      \ln\ln c_i-\ln\ln c_j=u\,(\ln (G_i)-\ln (G_j))+v\,(\ln(\omega_i)-\ln(\omega_j)),
  \]
  which corresponds to a hyperbola in the $(\alpha, \beta)$-plane. Finally, we note that the regions mirror the phase transitions from Thm. \ref{T_phase_transition}. Along an $(\omega,k)$-family, the asymptotics $G\sim K^{1/\omega}(\ln c)^{k/\omega}$ from the result give
  \[
      \ln q_C(a,b,c;\alpha,\beta)=\Bigl(1-\tfrac{k\beta}{\omega}\Bigr)\ln(\ln (c))
          -\alpha\beta\ln(\omega) -\frac{\beta}{\omega}\ln (K),
  \]
  such that the phase of the family has a critical exponent $\beta_0 = \omega/k$ is precisely the parameter at which the corresponding plane for the $(\omega, k)$-family drops out of the upper envelope in the geometric picture.

  Computationally, the upper envelope of $n$ planes has combinatorial complexity $O(n)$ and is
  obtained in $O(n\log n)$ time as a three-dimensional convex hull, after which the record-setter for each $(\alpha,\beta)$ or $(u, v)$ follows by planar point location in $O(\log n)$ \cite{deberg2008}. A single such computation resolves the entire parameter plane at once, whereas a naive sweep costs $O(n)$ per parameter value. We emphasize that this formalism is solely an organizing \textit{principle} rather than a generator of high-quality triples, as it consumes the output of any of Algorithms \ref{A_p_q}-\ref{A_cyclo} and returns the complete, exact record structure of the triples over $(\alpha,\beta)$, together with the minimal set of triples that are themselves optimal.

\iffalse
\section{Future Extensions to Other Quality Measures}
\label{S_Future}
\subsection{Divisor Quality}
\begin{definition}\label{D_DQ}\end{definition}
\subsection{Harmonic Mean Quality}
\begin{definition}\label{D_HQ}\end{definition}
\fi

\section{Conclusion}
We defined novel quality metrics that prioritize triples with small $\omega$, using the Doubly Geometric Mean (Def. \ref{D_DGM_Quality}). We analyzed the resulting quality metric to determine high-quality triples and asymptotics, developing characterizations of families yielding high-quality triples, including the Mersenne family (Theorem \ref{T_3_4}), the twin prime family (Theorem \ref{T_twin_DGM}), and fixed sequences of primes (Theorem \ref{T_main}). We also invoked Chen's Theorem (Thm. \ref{T_Chen}) to develop a proof of the asymptotic analogue of the $abc$-conjecture for our quality metric (Theorem \ref{T_Chen_DGM}). We also developed linear-time algorithms for determination of high-quality triples (Algs. \ref{A_2_3}, \ref{A_cyclo}).

We extended our quality metrics via a parametrization through the DGM quality class in Sec. \ref{S_DGM_Class} to modulate asymptotic behavior. Upon establishing phase transition results that determined asymptotics with respect to to parameter $\beta$ (Theorems \ref{T_phase_transition}, \ref{T_omega2_trichotomy}, \ref{T_uncond_LB}, \ref{T_cond_UB}), with various degrees of control on families of $abc$-triples and $\omega$ itself, we extended such results to develop asymptotics motivated by Szpiro bounds and relationships to the standard quality in the $abc$-conjecture itself (Theorems \ref{T-szpiro}, \ref{t-big}). A full proof of \ref{Smooth_class_ABC_Analogue} remains open, and finer results regarding bounds on the class $q_C(\cdot, \alpha, \beta)$ may even offer insight towards proofs of the conjecture itself, given the asymptotic bounds in Theorem \ref{t-big}.

Our unconditional resolution of the $abc$-analogue via Chen's theorem (Theorem \ref{T_Chen_DGM}) illustrates a principle that the DGM modification brings the analogue of the conjecture within reach of existing sieve methods, suggesting that the standard $abc$-conjecture itself may lie on a continuum of quality metrics whose more tractable members can be unconditionally settled, via the critical $\beta_0 = \omega/k$ measuring the exact damping required for each family. We view our phase-transition framework as evidence that families of modified quality metrics may be the natural objects of study, such that finer analysis of phase-space parameters $(\alpha, \beta)$, with the relevant factorization of the standard quality $q_s$, may yield further progress toward the $abc$-conjecture itself.

\bibliographystyle{unsrt}
\bibliography{numbertheory}

@article{elkies2007,
  title={{The ABC’s of number theory}},
  author={Elkies, Noam D},
  journal={The Harvard College Mathematics Review},
  volume={1},
  number={1},
  pages={57-76},
  year={2007}
}

@inproceedings{waldschmidt2015,
  title={{Lecture on the $abc$ conjecture and some of its consequences}},
  author={Waldschmidt, Michel},
  booktitle={Mathematics in the 21st Century: 6th World Conference, Lahore, March 2013},
  pages={211-230},
  year={2015},
  organization={Springer}
}

@misc{desmit2022,
  title={{ABC Triples}},
  author={de Smit, Bart},
  year={2020},
  howpublished={\url{https://pub.math.leidenuniv.nl/~smitbde/abc/}}
}

@article{nitaj1993-an,
  title={{An algorithm for finding good $abc$-examples}},
  author={Nitaj, Abderrahmane},
  journal={Comptes rendus de l'Acad{\'e}mie des sciences. S{\'e}rie 1, Math{\'e}matique},
  volume={317},
  number={9},
  pages={811-815},
  year={1993}
}

@article{nitaj1993-algorithms,
  title={{Algorithms for finding good examples for the $abc$ and Szpiro conjectures}},
  author={Nitaj, Abderrahmane},
  journal={Experimental Mathematics},
  volume={2},
  number={3},
  pages={223-230},
  year={1993},
  publisher={Taylor \& Francis}
}

@misc{nitaj2022,
  title={{The $abc$ conjecture home page}},
  author={Nitaj, Abderrahmane},
  year={2010},
  url={https://web.archive.org/web/20000819203144/http://www.math.unicaen.fr/~nitaj/abc.html}
}

@phdthesis{horst2010,
  title={{Finding ABC-triples using elliptic curves}},
  author={van der Horst, J},
  year={2010},
  type={Masters thesis},
  school={Universiteit Leiden}
}

@article{granville2002,
  title={{It’s as easy as $abc$}},
  author={Granville, Andrew and Tucker, Thomas},
  journal={Notices of the AMS},
  volume={49},
  number={10},
  pages={1224-1231},
  year={2002}
}

@article{oesterle1988,
  title={{Nouvelles approches du “th{\'e}oreme” de Fermat}},
  author={Oesterl{\'e}, Joseph},
  journal={Ast{\'e}risque},
  volume={161-162},
  pages={165-186},
  year={1988}
}

@inproceedings{masser1985,
  title={{Open problems}},
  author={Masser, David W},
  booktitle={Proceedings of the symposium on Analytic Number Theory, London, 1985},
  year={1985},
  organization={Imperial College}
}

@article{stewart2001,
  title={{On the $abc$ conjecture, II}},
  author={Stewart, Cameron L and Yu, Kunrui},
  year={2001},
  journal={Duke 
  Mathematical Journal}, volume={108}, number={1}, pages={169--181}, year={2001}, publisher={Duke University Press}
}

@article{martin2016,
  title={{$abc$ triples}},
  author={Martin, Greg and Miao, Winnie},
  journal={Functiones et Approximatio Commentarii Mathematici},
  volume={55},
  number={2},
  pages={145-176},
  year={2016},
  publisher={Adam Mickiewicz University, Faculty of Mathematics and Computer Science}
}

@article{wiles1995,
  title={{Modular elliptic curves and Fermat's last theorem}},
  author={Wiles, Andrew},
  journal={Annals of mathematics},
  volume={141},
  number={3},
  pages={443-551},
  year={1995},
}

@book{mason1984,
  title={{Diophantine equations over function fields}},
  author={Mason, Richard C},
  volume={96},
  year={1984},
  publisher={Cambridge University Press}
}

@article{stothers1981,
  title={{Polynomial identities and Hauptmoduln}},
  author={Stothers, W Wilson},
  journal={The Quarterly Journal of Mathematics},
  volume={32},
  number={3},
  pages={349-370},
  year={1981},
  publisher={Oxford University Press}
}

@article{browkin1994,
  title={{Some remarks on the $abc$-conjecture}},
  author={Browkin, Jerzy and Brzezi{\'n}ski, Juliusz},
  journal={Mathematics of Computation},
  volume={62},
  number={206},
  pages={931-939},
  year={1994}
}

@article{deweger1989,
  title={{Algorithms for Diophantine equations}},
  author={De Weger, Benne MM},
  journal={CWI tracts},
  volume={65},
  year={1989},
  publisher={Centrum voor wiskunde en informatica}
}

@article{tsang2010,
  title={{Fermat numbers}},
  author={Tsang, Cindy and Stein, William},
  journal={University of Washington},
  year={2010}
}

@article{mersenne,
  title={Mersenne primes: History, theorems and lists},
  author={Caldwell, Chris K},
  journal={Prime Pages},
  year={2014},
  url={https://primes.utm.edu/mersenne/}
}

@misc{mersennelist,
  title={{Great Internet Mersenne Prime Search (GIMPS). List Of Known Mersenne Prime Numbers}},
  year={2022},
  url={https://www.mersenne.org/primes/}
}

@article{pomerance1981,
  title={{Recent developments in primality testing}},
  author={Pomerance, Carl},
  journal={The Mathematical Intelligencer},
  volume={3},
  number={3},
  pages={97-105},
  year={1981},
  publisher={Springer}
}

@book{hardywright,
  title={{An introduction to the theory of numbers}},
  author={Hardy, Godfrey Harold and Wright, Edward Maitland and others},
  year={1979},
  publisher={Oxford university press}
}

@inproceedings{gnfs,
  title={{Factoring integers with the number field sieve}},
  author={Buhler, Joe P and Lenstra, Hendrik W and Pomerance, Carl},
  booktitle={The development of the number field sieve},
  pages={50-94},
  year={1993},
  organization={Springer}
}

@article{hardy1923some,
  author    = {Hardy, Godfrey H. and Littlewood, John E.},
  title     = {Some problems of `Partitio Numerorum'; III: On the expression of a number as a sum of primes},
  journal   = {Acta Mathematica},
  volume    = {44},
  number    = {1},
  pages     = {1-70},
  year      = {1923},
  publisher = {Springer},
  doi       = {10.1007/BF02403921}
}

@article{Chen1973,
  title = {On the representation of a large even integer as the sum of a prime and the product of at most two primes},
  author = {Chen, Jing Run},
  journal = {Scientia Sinica},
  volume = {16},
  pages = {157-176},
  year = {1973},
  publisher = {Science Press}
}

@article{mihailescu2004primary,
  author    = {Mih{\u{a}}ilescu, Preda},
  title     = {Primary cyclotomic units and a proof of {C}atalan's conjecture},
  journal   = {Journal f{\"u}r die reine und angewandte Mathematik (Crelles Journal)},
  volume    = {2004},
  number    = {572},
  pages     = {167-195},
  year      = {2004},
  publisher = {Walter de Gruyter},
  doi       = {10.1515/crll.2004.048},
  issn      = {0075-4102}
}

@book{vojta1987,
  title={Diophantine approximations and value distribution theory},
  author={Vojta, Paul},
  series={Lecture Notes in Mathematics},
  volume={1239},
  year={1987},
  publisher={Springer-Verlag},
  address={Berlin, New York},
  doi={10.1007/BFb0072989}
}

@article{gyory1979explicit,
  author    = {Gy{\H{o}}ry, K{\'a}lm{\'a}n},
  title     = {Explicit upper bounds for the solutions of some {D}iophantine equations},
  journal   = {Publicationes Mathematicae Debrecen},
  volume    = {26},
  number    = {1-2},
  pages     = {119-132},
  year      = {1979},
  publisher = {Institutum Mathematicum Universitatis Debreceniensis}
}

@book{silverman2009arithmetic,
  title     = {The Arithmetic of Elliptic Curves},
  author    = {Silverman, Joseph H.},
  series    = {Graduate Texts in Mathematics},
  volume    = {106},
  edition   = {2nd},
  year      = {2009},
  publisher = {Springer},
  address   = {Dordrecht},
  doi       = {10.1007/978-0-387-09494-6},
  isbn      = {978-0-387-09494-6},
  url       = {https://link.springer.com/book/10.1007/978-0-387-09494-6}
}

@article{frey1986links,
  author   = {Frey, Gerhard},
  title    = {Links between stable elliptic curves and certain {D}iophantine equations},
  journal  = {Annales Universitatis Saraviensis. Series Mathematicae},
  volume   = {1},
  number   = {1},
  pages    = {1-40},
  year     = {1986}
}

@article{szpiro1981proprietes,
  author    = {Szpiro, Lucien},
  title     = {Propri{\'e}t{\'e}s num{\'e}riques du faisceau dualisant relatif},
  journal   = {Ast{\'e}risque},
  volume    = {86},
  pages     = {44-78},
  year      = {1981},
  publisher = {Soci{\'e}t{\'e} math{\'e}matique de France},
  zbl       = {0517.14006}
}

@article{baker2014,
  author    = {Baker, Roger C. and Harman, Glyn},
  title     = {Shifted primes without large prime factors},
  journal   = {Acta Arithmetica},
  volume    = {83},
  number    = {4},
  pages     = {331-361},
  year      = {1998},
  doi       = {10.4064/aa-83-4-331-361}
}

@article{agrawal2004primes,
  title={PRIMES is in P},
  author={Agrawal, Manindra and Kayal, Neeraj and Saxena, Nitin},
  journal={Annals of Mathematics},
  pages={781-793},
  year={2004},
  publisher={JSTOR},
  doi={10.4007/annals.2004.160.781}
}

@article{specht1960theorie,
  author    = {Specht, Wilhelm},
  title     = {Zur Theorie der elementaren Mittel},
  journal   = {Mathematische Zeitschrift},
  volume    = {74},
  number    = {1},
  pages     = {91-98},
  year      = {1960},
  publisher = {Springer},
  doi       = {10.1007/BF01180630},
  language  = {german}
}

@book{deberg2008,
  title = {Computational Geometry: Algorithms and Applications},
  author = {de Berg, Mark and Cheong, Otfried and van Kreveld, Marc and Overmars, Mark},
  edition = {3rd},
  year = {2008},
  publisher = {Springer},
  address = {Berlin, Heidelberg},
  doi = {10.1007/978-3-540-77974-2},
  isbn = {978-3-540-77974-2}
}

\end{document}